\documentclass[12pt]{amsart}
\usepackage[sans]{dsfont}
\usepackage{amsfonts,amssymb,amsbsy,amsmath,amsthm,enumerate}

\numberwithin{equation}{section}

\newtheorem{theorem}{Theorem}[section]
\newtheorem{proposition}[theorem]{Proposition}
\newtheorem{lemma}[theorem]{Lemma}

\newtheorem{corollary}[theorem]{Corollary}

\theoremstyle{definition}

\usepackage{amsfonts}
\usepackage{amsmath}
\usepackage{bigints}
\usepackage{bm}
\usepackage{dsfont}
\usepackage{color}

\definecolor{gr}{rgb}   {0.,   0.69,   0.23 }

\def\beq{\begin{equation}}
\def\eeq{\end{equation}}
\newcommand{\bea}{\begin{eqnarray}}
\newcommand{\eea}{\end{eqnarray}}
\newcommand{\beas}{\begin{eqnarray*}}
\newcommand{\eeas}{\end{eqnarray*}}

\newcommand{\bel}{\begin{equation} \label}
\newcommand{\ee}{\end{equation}}

\newcommand{\bethl}{\begin{theorem} \label}

\newcommand{\beprl}{\begin{proposition} \label}
\newcommand{\epr}{\end{proposition}}

\newcommand{\belel}{\begin{lemma} \label}
\newcommand{\ele}{\end{lemma}}

\newcommand{\becol}{\begin{corollary} \label}
\newcommand{\eco}{\end{corollary}}

\newcommand{\bepf}{\begin{proof}}
\newcommand{\epf}{\end{proof}}

\newcommand{\pd}{\partial}

\newcommand{\xp}{x_{\bot}}
\newcommand{\xpa}{x_{\parallel}}
\newcommand{\ypa}{y_{\parallel}}
\newcommand{\xpe}{x_{\perp}}
\newcommand{\ype}{y_{\perp}}

\newcommand{\one}{\mathds{1}}

\newcommand{\pipe}{\pi_\perp}
\newcommand{\pipa}{\pi_\parallel}

\newcommand{\capa}{{{\rm Cap}}}
\newcommand{\lnd}{\ln_2}
\newcommand{\lnt}{\ln_3}
\newcommand{\rd}{{\mathbb R}^{2}}
\newcommand{\rt}{{\mathbb R}^{3}}

\newcommand{\re}{{\mathbb R}}

\newcommand{\Sbb}{{\mathbb S}}
\newcommand{\C}{{\mathbb C}}

\newcommand{\N}{{\mathbb N}}
\newcommand{\Z}{{\mathbb Z}}
\newcommand{\gB}{{\mathfrak{B}}}
\newcommand{\gC}{{\mathfrak{C}}}
\newcommand{\gD}{{\mathfrak{D}}}

\newcommand{\gM}{{\mathfrak{M}}}

\newcommand{\gp}{{\mathfrak{p}}}

\newcommand{\gS}{{\mathfrak{S}}}

\newcommand{\cB}{{\mathcal B}}

\newcommand{\cD}{{\mathcal D}}
\newcommand{\cE}{{\mathcal E}}

\newcommand{\cK}{{\mathcal K}}

\newcommand{\cM}{{\mathcal M}}
\newcommand{\cO}{{\mathcal O}}

\newcommand{\cR}{{\mathcal R}}

\newcommand{\cT}{{\mathcal T}}

\newcommand{\cW}{{\mathcal W}}

\newcommand{\veps}{{\varepsilon}}

\newcommand{\Hope} {H_{0,\perp}}
\newcommand{\Hopa} {H_{0,\parallel}}
\newcommand{\Ipe} {I_{\perp}}
\newcommand{\Ipa} {I_{\parallel}}

\newcommand{\omin}{\Omega_{\rm in}}
\newcommand{\omex}{\Omega_{\rm ex}}
\newcommand{\omj}{\Omega_j}
\newcommand{\Hpj}{H_{+,j}}
\newcommand{\Hmj}{H_{-,j}}

\newcommand{\srvpm}{V_\pm^{\frac{1}{2}}}
\newcommand{\oless}{\Omega_<}
\newcommand{\oneoless}{\one_{\Omega_<}}

\begin{document}
\title[SSF for geometric perturbations of magnetic Hamiltonians]{Threshold singularities of the spectral shift function for geometric perturbations of magnetic Hamiltonians}

\author[V.~Bruneau]{Vincent Bruneau}
\author[G. Raikov]{Georgi Raikov}

{
\begin{abstract}\setlength{\parindent}{0mm}
We consider the 3D Schr\"odinger operator $H_0$ with constant magnetic field $B$ of scalar intensity $b>0$, and its perturbations
  $H_+$ (resp., $H_-$) obtained by imposing Dirichlet (resp., Neumann) conditions on the boundary
  of the bounded domain $\Omega_{\rm in} \subset {\mathbb R}^3$. We introduce the Krein spectral shift functions
  $\xi(E;H_\pm,H_0)$, $E \geq 0$, for the operator pairs $(H_\pm,H_0)$, and study their singularities at the
  Landau levels $\Lambda_q : = b(2q+1)$, $q \in \Z_+$, which play the role of thresholds in the spectrum of  $H_0$. We show that $\xi(E;H_+,H_0)$ remains bounded as $E \uparrow \Lambda_q$,  $q \in \Z_+$,  being fixed, and obtain three asymptotic terms of $\xi(E;H_-,H_0)$ as $E \uparrow \Lambda_q$,  and of $\xi(E;H_\pm,H_0)$ as $E \downarrow \Lambda_q$. The first two terms are independent of the perturbation while the third one involves the {\em  logarithmic capacity} of the projection of
  $\Omega_{\rm in}$ onto the plane perpendicular to $B$.
\end{abstract}

\maketitle

{\bf  AMS 2010 Mathematics Subject Classification:} 35P20,  81Q10\\

{\bf  Keywords:} threshold singularities, spectral shift function, magnetic Laplacians, Dirichlet and Neumann boundary conditions, logarithmic capacity

\section{Introduction}
\label{s1} \setcounter{equation}{0}
Let
$$
B = (0,0,b), \quad b>0,
$$
be a vector in $\re^3$
which has the physical interpretation of a constant magnetic field. Then
 \bel{f1}
A(x) : = \frac{b}{2}\left(-x_2, x_1, 0\right), \quad x = (x_1,x_2,x_3) \in \re^3,
\ee
is a magnetic potential which generates $B$, i.e. ${\rm curl}\,A = B$,
$$
\Pi(A) = (\Pi_1(A),\Pi_2(A),\Pi_3(A)) : = -i\nabla - A
$$
is the magnetic gradient, and
$$
-\Delta_A : = \sum_{j=1}^3 \Pi_j(A)^2 =
\left(-i\frac{\partial}{\partial x_1} + \frac{b x_2}{2}\right)^2 + \left(-i\frac{\partial}{\partial x_2} - \frac{b x_1}{2}\right)^2 - \frac{\partial^2}{\partial x_3^2}
$$
is the magnetic Laplacian. In order to define the domain of an appropriate realization of $-\Delta_A$, self-adjoint in $L^2(\re^3)$, we need the following notations.
Let $\Omega$ be  an open non-empty subset of $\re^3$.
Introduce the magnetic Sobolev spaces
$$
{\rm H}_A^s(\Omega) : = \left\{u \in \cD'(\Omega) \,| \, \Pi(A)^\alpha u \in L^2(\re^3), \; \alpha \in \Z_+^3, \, 0 \leq |\alpha| \leq s\right\}, \quad s \in \Z_+.
 $$
Denote by ${\rm H}_{A,0}^s(\Omega)$ the closure of $C_0^\infty(\Omega)$ in the norm of ${\rm H}_A^s(\Omega)$ defined by
$$
\|u\|^2_{{\rm H}_A^s(\Omega)} : = \sum_{\alpha \in \Z_+^3: 0 \leq |\alpha| \leq s} \int_{\Omega} |\Pi(A)^\alpha u|^2\,dx.
$$
Then the operator $H_0 : = - \Delta_A$  with domain $\gD(H_0) : = {\rm H}_A^2(\re^3)$ is self-adjoint in $L^2(\re^3)$, and essentially self-adjoint on $C_0^\infty(\re^3)$ (see e.g.
\cite[Appendix]{GeMaSj91}).
 It is well known that
\bel{30}
\sigma(H_0) = \sigma_{\rm ac}(H_0) = [b,\infty),
\ee
and {\em the Landau levels}
$$
\Lambda_q : = b(2q+1), \quad q \in \Z_+ : = \left\{0,1,2,\ldots\right\},
$$
play the role of {\em  thresholds} in the spectrum $\sigma(H_0)$ of $H_0$
(see e.g. \cite{Fo28, La30}).\\
Next, as usual, we define a {\em domain} in $\re^d$, $d \geq 1$, as an open, connected, non-empty subset of $\re^d$.  Let  $\omin \subset \re^3$ be a bounded domain with boundary $\partial \omin \in C^\infty$. Set $$\Gamma := \partial \omin, \quad  \omex : = \re^3\setminus \overline{\omin}.$$
Then the operator $\Hpj : = -\Delta_A$, $j = {\rm ex,in}$, with domain
$$
\gD(\Hpj) : = \left\{u \in {\rm H}_A^2(\omj) \, | \, u_{|\Gamma} = 0\right\},
$$
is the {\em Dirichlet} realization of  $-\Delta_A$ on $\omj$.
 Similarly, if $\nu$ is the unit normal vector at $\Gamma$, outward looking with respect to $\omin$, then the operator $\Hmj : = -\Delta_A$, $j = {\rm ex,in}$, with
domain
$$
\gD(\Hmj) : = \left\{u \in {\rm H}_A^2(\omj) \, | \, \nu \cdot \Pi(A) u_{|\Gamma} = 0\right\},
$$
is the {\em Neumann} realization of  $-\Delta_A$ on $\omj$.
The operators $H_{\pm, j}$, $j={\rm ex,in}$, are self-adjoint in $L^2(\Omega_j)$. Moreover, $\Hpj$ (resp., $\Hmj$) corresponds to the closed quadratic form
    \bel{au1}
    \int_{\Omega_j} |\Pi(A)\, u|^2\,dx
    \ee
    with domain ${\rm H}_{A,0}^1(\Omega_j)$ (resp., ${\rm H}_{A}^1(\Omega_j)$). \\
    Using the orthogonal decomposition $L^2(\re^3) = L^2(\omin) \oplus L^2(\omex)$, set
$$
H_\pm : = H_{\pm, {\rm in}} \oplus H_{\pm, {\rm ex}}.
$$

    The aim of the article is to study the asymptotic behavior of the spectral shift functions $\xi(E;H_\pm,H_0)$ defined in the next section, as the energy $E$ approaches
    a given Landau level $\Lambda_q$, $q \in \Z_+$.\\
   The article is organized as follows. In Section \ref{s2} we introduce the spectral shift functions $\xi(E;H_\pm,H_0)$ and describe their main properties. In Section \ref{fs1} we state our main result, Theorem \ref{th1}, and briefly comment on it. In Section \ref{supersection} we prove several important auxiliary results, Propositions \ref{p10}, \ref{p8}, and \ref{fp1}, while the proof of Theorem \ref{th1} can be found in Section \ref{s5}. Finally, the Appendix contains the details concerning some technical results used in the main text of the article.

 \section{The spectral shift function}
\label{s2}
 Let $X$ be a separable Hilbert space. Denote by $\gB(X)$ (resp., $\gS_\infty(X)$) the class of linear bounded (resp., compact) operators acting in $X$, and by $\gS_p(X)$, $p \in [1,\infty)$, the $p$th Schatten-von Neumann space of operators $T \in \gS_\infty(X)$ for which the norm
    $$
    \|T\|_p : = \left({\rm Tr}\,(T^* T)^{p/2}\right)^{1/p}
    $$
    is finite. In particular, $\gS_1(X)$ is the trace class, and $\gS_2(X)$ is the Hilbert-Schmidt class over $X$. If $X = L^2(\re^3)$, we omit $X$ in the notations $\gB(X)$ and $\gS_p(X)$, $p \in [1,\infty]$ .\\
    By the Dirichlet-Neumann bracketing and the non-negativeness of the quadratic form \eqref{au1}, we have
    \bel{32}
H_+ \geq H_0 \geq H_- \geq 0.
    \ee
     By \eqref{30},
and $b > 0$, we find that the operators $H_0$, and hence $H_+$, are invertible. It is not difficult to see that $H_-$ is invertible as well. To this end, arguing as in the proof of Proposition \ref{p10} below, we find that
$$
 (H_- + I)^{-1} - (H_0 + I)^{-1} \in \gS_2 \subset \gS_\infty.
$$
Therefore, the Weyl theorem on the invariance of the essential spectrum under relatively compact perturbations yields
    $$
    \sigma_{\rm ess}(H_-) =  \sigma_{\rm ess}(H_0) = [b,\infty).
    $$
    Hence, if $0 \in \sigma(H_-)$, then the zero should be a discrete eigenvalue of $H_-$. Let $u \in \gD(H_-)$ such that $H_- u = 0$. By  \eqref{au1}, we have
    \bel{33}
    \Pi(A)u_{|\omin} = 0, \quad \Pi(A)u_{|\omex} = 0.
    \ee
    Taking into account the explicit expression \eqref{f1} for $A$, we find that the only element $u \in \gD(H_-)$ which satisfies \eqref{33}, is $u=0$, and hence $0 \not \in \sigma(H_-)$.\\
    Further, \eqref{32} implies
    \bel{34}
H_-^{-1} \geq H_0^{-1} \geq H_+^{-1}.
    \ee
    Set
    $$
    V_+: = H_0^{-1} - H_+^{-1}, \quad V_-: = H_-^{-1} - H_0^{-1}.
    $$
    Then, \eqref{34} yields $V_\pm \geq 0$.\\

    \beprl{p10} We have
    \bel{46}
    V_\pm \in \gS_2.
    \ee
    Moreover,
    \bel{47}
    H_\pm^{-2} - H_0^{-2} \in \gS_1.
    \ee
    \epr
    The proof of Proposition \ref{p10} can be found in Section \ref{s2a}.\\

    {\em Remark}: In \cite{Bi62, BiSo79, BiSo80a}, the authors consider second-order elliptic differential operators in $\re^d$, $d \geq 2$, equip them with Dirichlet or Neumann boundary conditions on appropriate hypersurfaces, and obtain results closely related to our Proposition \ref{p10}. Although, formally, our operator $H_0$ is not in the classes of the operators considered in
    \cite{Bi62, BiSo79, BiSo80a}, the methods applied there may improve relations \eqref{46} and \eqref{47} which, nonetheless,
    are sufficient for the purposes of this article. \\

Using \eqref{47}, we define the spectral shift function (SSF) $\xi(E;H_\pm,H_0)$ as
    $$
    \xi(E;H_\pm,H_0) : = \left\{
    \begin{array} {l}
    -\xi(E^{-2};H_\pm^{-2},H_0^{-2}) \quad \mbox{if} \quad E > \inf \sigma(H_\pm),\\[2mm]
    0 \quad \mbox{if} \quad E < \inf \sigma(H_\pm),
    \end{array}
    \right.
    $$
    where, for almost every $E>0$,
    \bel{37a}
    \xi(E^{-2};H_\pm^{-2},H_0^{-2}) : = \frac{1}{\pi} \lim_{\varepsilon \downarrow 0}\;{\rm arg}\;{\rm Det}\,\left(\left(H_\pm^{-2} - E^{-2} - i\varepsilon\right)\left(H_0^{-2} - E^{-2} - i\varepsilon\right)^{-1}\right),
    \ee
    the branch of the argument being fixed by the condition
    $$
    \lim_{{\rm Im}\,z \to \infty} {\rm arg}\;{\rm Det}\,\left(\left(H_\pm^{-2} - z\right)\left(H_0^{-2} - z\right)^{-1}\right) = 0
    $$
    (see the original work \cite{Kr53} or \cite[Chapter 8]{Ya92}). The SSF $\xi(\cdot;H_\pm,H_0)$ is the unique element of $L_{\rm loc}^1(\re)$ which satisfies the {\em Lifshits-Krein identity}
    $$
    {\rm Tr}\,\left(f(H_\pm) - f(H_0)\right) = \int_\re f'(E) \, \xi(E;H_\pm,H_0) \, dE, \quad f \in C_0^\infty(\re),
    $$
    and the normalization condition
    $$
    \xi(E;H_\pm,H_0) = 0, \quad E < \inf \sigma(H_\pm).
    $$
    Since $\inf \sigma(H_\pm) > 0$, so that $\xi(E;H_\pm,H_0) = 0$ for $E \in (-\infty,0]$, in the sequel we will consider $\xi(E;H_\pm,H_0)$ only for $E>0$. \\
    For almost every $E \in [b, \infty) = \sigma_{\rm ac}(H_0)$, the {\em  Birman-Krein formula} implies
    $$
    {\rm det}\, S(E; H_\pm, H_0) = e^{-2\pi i\xi(E;H_\pm,H_0)}
    $$
   where $S(E; H_\pm, H_0)$ is the {\em  scattering matrix} for the operator pair $(H_\pm,H_0)$  (see \cite{BiKr62} or \cite[Chapter 8]{Ya92}). On the other hand, for almost every $E \in (0,b)$ we have
     \bel{m14}
    \xi(E;H_-,H_0) = - {\rm Tr}\,\one_{(-\infty, E)}(H_-).
   \ee
   Here and in the sequel $\one_S$ denotes the characteristic function of the set $S$. Thus, $\one_{S}(T)$ is the spectral projection of $T$ corresponding to the Borel set $S \subset \re$, and by \eqref{m14}
    $ -\xi(E;H_-,H_0) $ is equal to the number of the eigenvalues of $H_-$ less than $E$ and counted with the multiplicities. \\
    {\em A priori},  the SSF $\xi(E;H_\pm,H_0)$ is defined only for almost every  $E \in \re$.
    Our next goal is to introduce a canonic representative of the class of equivalence $\xi(\cdot;H_\pm,H_0)$ following the main ideas of \cite{Pu97} (see below Proposition \ref{fp1}). Let $\C_\pm : = \{z \in \C \, | \, \pm {\rm Im}\,z > 0\}$. For $z \in \C_-$ set
    $$
    T^\pm(z) : = \srvpm (H_0^{-1} - z^{-1})^{-1} \srvpm.
    $$

    \beprl{p8} Let $E \in (0,\infty) \setminus b(2\Z_+ + 1)$. Then there exists a norm limit
    \bel{6}
    T^\pm(E): = {\rm n}-\lim_{\C_- \ni z \to E}  T^\pm(z)\in \gS_2,
    \ee
    and
    \bel{38}
    {\rm Im}\, T^\pm(E)  \in \gS_1.
    \ee
    Moreover,  ${\rm Re}\, T^\pm(E)$ {\rm (}resp., $ {\rm Im}\, T^\pm(E)${\rm )} depends continuously in $\gS_2$ (resp., in $\gS_1)$ on $E \in (0,\infty) \setminus b(2\Z_+ + 1)$.
    \epr

    The proof of Proposition \ref{p8} can be found in Subsection \ref{s3}. \\
Let $T = T^*$ be a compact operator in a Hilbert space. For $s>0$ set
    $$
    n_\pm(s;T) = {\rm Tr}\,\one_{(s,\infty)}(\pm T).
    $$
    Thus $n_+(s,T)$ (resp., $n_-(s,T)$) is just the number of the eigenvalues of $T$ counted with the multiplicities, greater than $s>0$ (resp., less than $-s < 0$). \\
    For $E \in (0,\infty)\setminus b(2\Z_+ + 1)$ set
    \bel{39}
\tilde{\xi}(E; H_\pm, H_0) : = \pm \frac{1}{\pi} \int_\re n_\pm \left(1; {\rm Re} \,  T^\pm(E) + t \, {\rm Im}\, T^\pm(E)\right) \, \frac{dt}{1+t^2}.
     \ee
   \beprl{fp1}
    The function $\tilde{\xi}(\cdot; H_\pm, H_0)$ is well defined on $(0,\infty)\setminus b(2\Z_+ +1 )$, bounded on every compact subset of $(0,\infty)\setminus b(2\Z_+ +1 )$, and  continuous on $(0,\infty)\setminus (\sigma_p(H_\pm) \cup b(2\Z_+ + 1))$ where $\sigma_p(H_\pm)$ denotes the set of the eigenvalues of $H_\pm$.\\ Moreover, for almost every $E \in (0,\infty)$ we have
    \bel{39a}
    \xi(E; H_\pm, H_0) = \tilde{\xi}(E; H_\pm, H_0).
    \ee
    \epr
    The proof of Proposition \ref{fp1} can be found in Subsection \ref{fs2}.\\

    {\em Remark}: In view of Proposition \ref{fp1},  we  identify in the sequel the SSF $\xi(E; H_\pm, H_0)$ with $\tilde{\xi}(E; H_\pm, H_0)$, and  assume that it is defined for every $E \in (0,\infty)\setminus b(2\Z_+ + 1)$.

   \section{Main Results}
   \label{fs1}
   Let $\cE \subset \rd$  be a Borel set, and $\gM(\cE)$ denote the set of compactly supported probability measures on $\cE$. Then  the  {\em logarithmic capacity} of $\cE$  is defined as
$\capa(\cE) : = e^{-{\mathcal I}(\cE)}$ where
$$
{\mathcal I}(\cE) : = \inf_{\mu \in \gM(E)}  \int_{\cE} \int_{\cE} \ln{|x-y|^{-1}} d\mu(x) d\mu(y).
$$
The properties of $\capa(\cE)$ we need, are summarized in Subsection \ref{ss55}.  A systematic exposition of the theory of the logarithmic capacity can be found, for example, in \cite[Chapter 5]{Ran95} and \cite[Chapter II, Section 4]{La72}. \\
Let $\cE \subset \rd$ be a Borel set such that $\capa(\cE) \in (0,\infty)$. Set
\bel{ms50}
    \gC(\cE) : = 1 + \ln{\left(b\,\capa(\cE)^2\right)}.
    \ee
Note that if $\cE$ is a bounded domain, then $\capa(\cE) \in (0,\infty)$. \\
    For $x \in \re^3$, we write $x = (\xpe,\xpa)$ where $\xpe = (x_1,x_2) \in \rd$ are the variables on the plane perpendicular to the magnetic field $B$ while $\xpa = x_3 \in \re$ is the variable along $B$.
     For $x = (\xpe,\xpa) \in \re^3$ define the projections $\pipe(x): = \xpe$, $\pipa(x) : = \xpa$. Note that if $\Omega \subset \re^3$ is a (bounded) domain, then $\pipe(\Omega) \subset \rd$ is a (bounded) domain as well.
    Set $$\cO_{\rm in} : = \pipe(\Omega_{\rm in}).$$
    Thus, $\cO_{\rm in}$ is the projection of the obstacle $\Omega_{\rm in}$ onto the plane perpendicular to the magnetic field $B$.\\
    For $\lambda > 0$ small enough, and $C \in \re$ set
    $$
     \lnd(\lambda) : = \ln{|\ln{\lambda}|}, \quad \lnt(\lambda) : = \ln{\lnd(\lambda)},
    $$
    and
    $$
    \Phi_0(\lambda) : = \frac{|\ln{\lambda}|}{\lnd(\lambda)}, \quad \Phi_1(\lambda;C) : = \Phi_0(\lambda) \left(1 + \frac{\lnt(\lambda)}{\lnd(\lambda)} + \frac{C}{\lnd(\lambda)}\right).
    $$

    \begin{theorem} \label{th1}
    Let $\Omega_{\rm in}$ be a bounded domain with $\partial \Omega_{\rm in} \in C^\infty$. Fix $q \in \Z_+$. Then,
    \bel{ms1}
    \xi(\Lambda_q-\lambda; H_+, H_0) = O(1),
    \ee
    \bel{m13}
    \xi(\Lambda_q - \lambda; H_-,H_0) = - \frac{1}{2}\, \Phi_1(\lambda; \gC(\cO_{\rm in})) + o\left(\frac{|\ln{\lambda}|}{\lnd(\lambda)^2}\right),
    \ee
     \bel{ms2}
     \xi(\Lambda_q + \lambda; H_\pm,H_0) = \pm \frac{1}{4}\, \Phi_1(\lambda; \gC(\cO_{\rm in})) + o\left(\frac{|\ln{\lambda}|}{\lnd(\lambda)^2}\right),
    \ee
    as $\lambda \downarrow 0$.
    \end{theorem}

    {\em Remarks}: (i) Evidently, \eqref{ms1} and \eqref{ms2} with sign ``+" imply
    \bel{ms9}
    \lim_{\lambda \downarrow 0} \frac{\xi(\Lambda_q-\lambda; H_+, H_0)}{\xi(\Lambda_q+\lambda; H_+, H_0)} = 0,
    \ee
   while  \eqref{m13} and \eqref{ms2} with sign ``--" imply
    \bel{ms10}
    \lim_{\lambda \downarrow 0} \frac{\xi(\Lambda_q-\lambda; H_-, H_0)}{\xi(\Lambda_q+\lambda; H_-, H_0)} = 2.
    \ee
     In a certain sense, relations \eqref{ms9} and \eqref{ms10} can be considered as generalizations of the classical Levinson theorem (see the original work \cite{Le49} or the survey article \cite{Ro99}), which relates  the (finite) number  of the negative eigenvalues of the non-magnetic Schr\"odinger operator $-\Delta + V$ with electric potential $V$ which decays fast enough at infinity, and the limit $\lim_{E \downarrow 0} \xi(E; -\Delta + V, -\Delta)$ where $\xi(E; -\Delta + V, -\Delta)$ is the SSF for the operator pair $(-\Delta + V, -\Delta)$.\\
     (ii) By the so-called telescopic property of the SSF, we have
     $$
     \xi(E; H_+,H_-) = \xi(E; H_+,H_0) - \xi(E; H_-,H_0), \quad E \in (0,\infty)\setminus b(2\Z_+ + 1).
     $$
     Therefore, \eqref{ms1} - \eqref{m13} imply
      $$
    \xi(\Lambda_q - \lambda; H_+,H_-) = \frac{1}{2}\, \Phi_1(\lambda; \gC(\cO_{\rm in})) + o\left(\frac{|\ln{\lambda}|}{\lnd(\lambda)^2}\right), \quad \lambda \downarrow 0,
    $$
     while \eqref{ms2} implies
     $$
     \xi(\Lambda_q + \lambda; H_+,H_-) =  \frac{1}{2}\, \Phi_1(\lambda; \gC(\cO_{\rm in})) + o\left(\frac{|\ln{\lambda}|}{\lnd(\lambda)^2}\right), \quad \lambda \downarrow 0.
    $$
    In particular, similarly to \eqref{ms9}-\eqref{ms10}, we have
    $$
    \lim_{\lambda \downarrow 0} \frac{\xi(\Lambda_q-\lambda; H_+, H_-)}{\xi(\Lambda_q+\lambda; H_+, H_-)} = 1.
    $$
     (iii) According to \eqref{m14},  we have
    \bel{m15}
    \xi(\Lambda_0-\lambda; H_-,H_0) = - {\rm Tr}\,\one_{(-\infty, \Lambda_0-\lambda)}(H_-) =
    \ee
    $$
    - {\rm Tr}\, \one_{(-\infty, \Lambda_0-\lambda)}(H_{-, \rm{ex}})  - {\rm Tr}\,\one_{(-\infty, \Lambda_0-\lambda)}(H_{-, \rm{in}}), \quad \lambda>0.
     $$
    Since the operator $H_{-, \rm{in}}$ is a second-order elliptic partial differential operator acting in a bounded domain with smooth boundary, its spectrum $\sigma(H_{-, \rm{in}})$ is discrete, and
    $$
    {\rm Tr}\,\one_{(-\infty, \Lambda_0-\lambda)}(H_{-, \rm{in}}) = O(1), \quad \lambda \downarrow 0.
    $$
    Then, \eqref{m13} with $q=0$  implies
    $$
    {\rm Tr}\,\one_{(-\infty, \Lambda_0-\lambda)}(H_{-, \rm{ex}}) =  \frac{1}{2}\, \Phi_1(\lambda; \gC(\cO_{\rm in})) + o\left(\frac{|\ln{\lambda}|}{\lnd(\lambda)^2}\right), \quad \lambda \downarrow 0,
    $$
    which describes the accumulation  of the discrete spectrum of the exterior Neumann magnetic Laplacian $H_{-, \rm{ex}}$ at $\Lambda_0 = \inf \sigma_{\rm ess}(H_{-, \rm{ex}})$. \\

    Let us compare Theorem \ref{th1} with similar results available in the literature.  The threshold singularities of the SSF for the operator pair $(H_0 + V, H_0)$ where $V$ is a real-valued fast decaying {\em electric potential}, were considered in \cite{FeRa04}. The cases of $V$ of power-like decay, exponential decay, and compact support were handled. Formally, our Theorem \ref{th1} resembles the results of \cite{FeRa04} on compactly supported $V$, which however are less precise than \eqref{m13} and \eqref{ms2}: the right-hand side of the analogue of \eqref{m13} (resp., of \eqref{ms2}) in \cite{FeRa04} is $-\frac{1}{2} \Phi_0(\lambda)(1+o(1))$ (resp., $\pm\frac{1}{4} \Phi_0(\lambda)(1+o(1))$).\\[2mm]
    A problem closely related to the analysis of the SSF $\xi(\cdot; H_0 + V, H_0)$ as $E \to \Lambda_q$ for a given $q \in \Z_+$, is  the investigation of {\em accumulation of
    resonances} of $H_0 + V$ at $\Lambda_q$ performed in \cite{BoBrRa07, BoBrRa14a, BoBrRa14b}. The asymptotic distribution of resonances near the Landau levels for the operators $H_\pm$ considered in this article, is studied in \cite{BrSa16a}. \\
    Let us mention also some 2D results related to Theorem \ref{th1}. It is well known that in the 2D case the spectrum of the Landau Hamiltonian is purely point and consists of the Landau levels which are eigenvalues of infinite multiplicity (see \eqref{m10} -- \eqref{m10a} below). Hence, the problem of the singularities of the SSF for the 2D analogue of the operator pair $(H_\pm, H_0)$ reduces to the study of the accumulation of the discrete eigenvalues of the 2D analogues of $H_\pm$ at the Landau levels. Such a study was undertaken in \cite{PuRo07} for the Dirichlet case, in \cite{Pe09, GoKaPe16} for the Neumann case, and in \cite{GoKaPe16} for Robin boundary conditions. \\

    \section{Proofs of the auxiliary results}
    \label{supersection}
    \subsection{Proof of Proposition \ref{p10}}
    \label{s2a}
    We start with the following key

\belel{l1}
Let $\omega \in C_0^\infty(\re^3; [0,1])$ such that $\omega = 1$ in a vicinity of $\Gamma$. Then we have
    \bel{2}
    \srvpm = \srvpm H_0 \omega H_0^{-1}.
    \ee
    \ele
    \bepf
    Let $P_\pm$ be the orthogonal projection onto $({\rm Ker}\,V_\pm)^\perp$. Then, $\srvpm = \srvpm P_\pm$. Set $\widetilde{\omega} : = 1 - \omega$. Note that $\widetilde{\omega}$ vanishes in vicinity of $\Gamma$. We have
    $$
    \srvpm = \srvpm P_\pm H_0 (\omega + \widetilde{\omega}) H_0^{-1}.
    $$
    Therefore, in order to prove \eqref{2}, it suffices to show that
    \bel{1}
    P_\pm H_0 \widetilde{\omega} H_0^{-1} = 0.
    \ee
    Define the operator $H_{00} : = -\Delta_A$ with  domain
    $$
    \gD(H_{00}) : = \left\{u \in {\rm H}_A^2(\re^3) \, | \, u_{|\Gamma} = \nu \cdot \Pi(A)u_{|\Gamma} = 0\right\}.
    $$
    Thus the operators $H_0$, $H_+$, and $H_-$ are extensions of the operator $H_{00}$. If $u \in L^2(\re^3)$, then
    $\widetilde{\omega} H_0^{-1} u \in \gD(H_{00})$ and $H_j \,\widetilde{\omega} H_0^{-1} u = H_{00} \,\widetilde{\omega} H_0^{-1} u$, $j=0,+,-$.
    Therefore, $V_\pm H_0 \,\widetilde{\omega} H_0^{-1} u = 0$, i.e. $H_0 \,\widetilde{\omega} H_0^{-1} u \in {\rm Ker}\, V_\pm$ which implies that \eqref{1} holds true.
    \epf

    Further, we note that
    \bel{3}
    H_0 \omega H_0^{-1}  = \omega + [H_0,\omega] H_0^{-1}
    \ee
     and obtain a convenient representation of the commutator
    $[H_0,\omega]$.\\
    To this end, we introduce the Landau Hamiltonian $\Hope$, i.e. the 2D Schr\"odinger operator with constant scalar magnetic field $b>0$,
\bel{D132}
\Hope =
\left(-i\frac{\partial}{\partial x_1} +  \frac{b x_2}{2}\right)^2  +
\left(-i\frac{\partial}{\partial x_2} -  \frac{b x_1}{2}\right)^2, \quad \xpe = (x_1,x_2) \in \rd,
\ee
essentially self-adjoint on $C_0^\infty(\rd)$, and self-adjoint in $L^2(\rd)$.
We have
$$
\Hope = a^* a + b
$$
where
    \bel{f19}
a^* = -2i e^{\phi} \frac{\partial}{\partial \zeta} e^{-\phi} = -2i \left(\frac{\partial}{\partial \zeta} - \frac{\partial \phi}{\partial \zeta}\right), \quad \zeta = x_1 + i x_2,
    \ee
is the magnetic creation operator,
$$
a = -2i e^{-\phi} \frac{\partial}{\partial \bar{\zeta}} e^{\phi }= -2i \left(\frac{\partial}{\partial \bar{\zeta}} + \frac{\partial \phi}{\partial \bar{\zeta}}\right), \quad \bar{\zeta} = x_1-ix_2,
$$
is the magnetic annihilation operator,
and $\phi(\xpe) : = \frac{b|\xp|^2}{4}$, $\xpe \in \rd$, so that $\Delta \phi = b$. \\
The operators $a$ and $a^*$ are closed on their common domain $\gD(a) = \gD(a^*) = \gD(\Hope^{1/2})$, they are mutually adjoint in $L^2(\rd)$, and satisfy
$$
[a, a^*] = 2b.
	$$
 It is well known that
    \bel{m10}
 \sigma(\Hope) = \bigcup_{j \in \Z_+} \left\{\Lambda_j\right\},\\[2mm]
    \ee
 $$
  {\rm Ker}\,(\Hope - \Lambda_j) = (a^*)^j {\rm Ker}\,a, \quad j \in \Z_+,
  $$
  $$
  {\rm Ker}\,a : = \left\{u \in L^2(\rd) \, | \, u = g e^{-\phi}, \quad \frac{\partial g}{\partial \overline{\zeta}} = 0\right\},
  $$
  and, accordingly,
  \bel{m10a}
  {\rm dim \; Ker}\,(\Hope - \Lambda_j) = \infty, \quad j \in \Z_+.
  \ee
Denote by $p_j$ the orthogonal projection onto $ {\rm Ker}\,(\Hope - \Lambda_j)$,  $j \in \Z_+$.
Next, set
$$
\Hopa : = - \frac{d^2}{d\xpa^2}, \quad \gD(\Hopa) = {\rm H}^2(\re).
$$
Then we have
$$
H_0 = \Hope \otimes \Ipa + \Ipe \otimes \Hopa
$$
where $\Ipe$ and $\Ipa$ are the identities in $L^2(\rd_{\xpe})$ and $L^2(\re_{\xpa})$ respectively, and a simple calculation implies the following
\belel{l2}
Let $\omega \in C_0^\infty(\rd;\re)$. Then we have
    \bel{4}
    K(\omega) : = [H_0,\omega]  = -\Delta \omega + \sum_{j=1}^3 \omega_j G_j = \Delta \omega + \sum_{j=1}^3 G_j \omega_j = - K(\omega)^*,
    \ee
    where
    $$
    \omega_1 : = - 2i \frac{\partial \omega}{\partial \bar{\zeta}}, \quad \omega_2 : = - 2i \frac{\partial \omega}{\partial \zeta}, \quad \omega_3 : = - 2 \frac{\partial \omega}{\partial \xpa},
    $$
    $$
    G_1 : = a^* \otimes \Ipa, \quad G_2 : = a \otimes \Ipa, \quad G_3 : = \Ipe \otimes \pd,
    $$
    and
    $$
    \pd: = \frac{d}{d\xpa}, \quad \gD(\pd) : = {\rm H}^1(\re).
    $$
    \ele
    Note that ${\rm supp}\,\Delta \omega \subset {\rm supp}\,\omega$ and ${\rm supp}\,\omega_j \subset {\rm supp}\,\omega$, $j=1,2,3$. Moreover,
    the operators $K(\omega) H_0^{-1}$ and, hence, $H_0^{-1} K(\omega)^*$ are compact in $L^2(\re^3)$.

    \belel{l4}
    Let $\eta \in C_0^\infty(\re^3)$, $j=0,+,-.$\\
    {\rm (i)} We have
    \bel{49}
    \eta H_j^{-1/2} \in \gS_4,
    \ee
    and hence
    \bel{49a}
    \eta H_j^{-1} \eta \in \gS_2.
    \ee
    {\rm (ii)} Moreover,
    \bel{50}
    \eta H_j^{-1} \in \gS_2.
    \ee
    \ele
    \bepf
    The validity of \eqref{49} and \eqref{50} follows easily from the diamagnetic inequality (see e.g. \cite{AvHeSi78} and \cite{HuLeMuWa01a}), and the results of \cite{BiSo80} concerning the spectral properties of elliptic non-magnetic differential operators.
    \epf

    Now we are in position to prove Proposition \ref{p10}. As above, let $\omega \in C_0^\infty(\re^3; [0,1])$ satisfy $\omega = 1$ in a vicinity of $\Gamma$, and let $\eta \in C_0^\infty(\re^3; [0,1])$ satisfy $\eta = 1$ in a vicinity of ${\rm supp}\,\omega$. By Lemma \ref{l1} and \eqref{3}, we have
    $$
    V_\pm = (H_0^{-1} K^* + \omega) \eta V_\pm \eta (\omega +  K H_0^{-1}).
    $$
    Since $\eta V_\pm \eta \in \gS_2$ by \eqref{49a}, and the operators $H_0^{-1} K^* + \omega$ and $\omega +  K H_0^{-1}$ are bounded, we obtain \eqref{46}. 
    Let us now prove \eqref{47}. Write
    \bel{fin1}
    H_\pm^{-2} - H_0^{-2} = \mp V_\pm (\omega +  K(\omega) H_0^{-1}) H_0^{-1} \mp H_\pm^{-1} (H_0^{-1} K^*(\omega) + \omega) V_\pm.
    \ee
    Let us show that
    \bel{fin7}
    V_\pm (\omega +  K(\omega) H_0^{-1}) H_0^{-1} \in \gS_1.
    \ee
    By \eqref{46} we have $V_\pm \in \gS_2$, \eqref{50} implies $\omega H_0^{-1} \in \gS_2$, and therefore
    \bel{fin2}
    V_\pm \omega H_0^{-1} \in \gS_1.
    \ee
    Further, let $\theta \in C_0^\infty(\rt; [0,1])$ satisfy $\theta = 1$ on ${\rm supp}\,\omega$. Then, by \eqref{4}, we have
    \bel{fin3}
    V_\pm K(\omega) H_0^{-2} = V_\pm K(\omega) \theta H_0^{-2} = V_\pm K(\omega) H_0^{-1} \theta H_0^{-1} + V_\pm K(\omega) H_0^{-1} K(\theta) H_0^{-2}.
    \ee
    Since $V_\pm, \theta H_0^{-1} \in \gS_2$, and $K(\omega) H_0^{-1}$ is bounded, we get
    \bel{fin4}
    V_\pm K(\omega) H_0^{-1} \theta H_0^{-1} \in \gS_1.
    \ee
    Further, by \eqref{4}, we have
    \bel{fin5}
    V_\pm K(\omega) H_0^{-1} K(\theta) H_0^{-2} = V_\pm K(\omega) H_0^{-1} \left(\Delta \theta + \sum_{j=1}^3 G_j \theta_j\right) H_0^{-2}.
    \ee
    Since $V_\pm, \Delta \theta\, H_0^{-1}, \theta_j H_0^{-1} \in \gS_2$, while the operators $K(\omega) H_0^{-1}$, $K(\omega) H_0^{-1} G_j$ are bounded,
    we find that \eqref{fin5} yields $ V_\pm K(\omega) H_0^{-1} K(\theta) H_0^{-2} \in \gS_1$ which combined with \eqref{fin2}, \eqref{fin3}, and \eqref{fin4} implies \eqref{fin7}. In a similar manner we prove that
    \bel{fin8}
    H_\pm^{-1} (H_0^{-1} K^*(\omega) + \omega) V_\pm \in \gS_1.
    \ee
    Putting together \eqref{fin1}, \eqref{fin7}, and \eqref{fin8}, we obtain \eqref{47}.

\subsection{Proof of Proposition \ref{p8}}
\label{s3}

   Let $z \in \C_-$. Combining \eqref{2} and \eqref{3} with  \eqref{4}, we find that
    \bel{7}
    T^\pm(z) =
    \srvpm (\omega + K(\omega)H_0^{-1}) (H_0^{-1} - z^{-1})^{-1} (\omega + H_0^{-1} K(\omega)^*) \srvpm.
    \ee
    Evidently,
    \bel{8}
    (H_0^{-1} - z^{-1})^{-1} = -z^2 (H_0 - z)^{-1} - z,
    \ee
    \bel{au2}
    H_0^{-1} (H_0 -z)^{-1} = (H_0 -z)^{-1} H_0^{-1} = \frac{1}{z} (H_0-z)^{-1} - \frac{1}{z} H_0^{-1},
    \ee
   \bel{au3}
     H_0^{-1} (H_0 -z)^{-1} H_0^{-1} = \frac{1}{z^2}  (H_0-z)^{-1} - \frac{1}{z^2} H_0^{-1} - \frac{1}{z} H_0^{-2}.
     \ee
     Combining \eqref{7} with \eqref{8} -- \eqref{au3}, and taking into account that $\srvpm (\omega + K H_0^{-1}) = \srvpm$, we get
    \bel{9}
     T^\pm(z) = M^\pm_1(z) + R^\pm_1(z)
     \ee
     where the {\em main term} is
     \bel{9a}
     M^\pm_1(z) : = -\srvpm (z\omega + K)  (H_0 - z)^{-1}  (z\omega + K^*) \srvpm,
     \ee
     while the {\em rest} is
     $$
     R^\pm_1(z) : =
    z\srvpm (\omega   H_0^{-1}  K^* +  K   H_0^{-1} \omega  +  K   H_0^{-2}  K^* - I) \srvpm  + \srvpm K   H_0^{-1}  K^* \srvpm.
     $$
    Since $R^\pm_1$ extends to an affine function form $\C$ to $\gS_2$, we obtain the following elementary
 \beprl{p13}
     For every $E \in \re$ there exists $R^\pm_1(E) = R^\pm_1(E)^* \in \gS_2$ such that
     $$
   \lim_{\C_- \ni z \to E} \|R^\pm_1(z) - R^\pm_1(E)\|_2 = 0,
    $$
     $R^\pm_1(E)$ depends continuously in $\gS_2$ on $E$, and
    $$
    \|R^\pm_1(E)\|_2 = O(|E| + 1), \quad E \in \re.
    $$
     \epr

    Set $P_j : = p_j \oplus \Ipa$, $j \in \Z_+$. For a given $q \in \Z_+$ put
    $$
    P^{\leq}_q : = \sum_{j\leq q} P_j, \quad P_q^{>}: = \sum_{j > q} P_j.
    $$
    Thus, $P_q^{\leq}$ and $P_q^>$ are orthogonal projections in $L^2(\re^3)$, and  $P_q^{\leq} + P_q^> = I$. Taking into account \eqref{9a}, we find that
    $$
    M^\pm_1(z) = M^\pm_2(z) + R^\pm_2(z)
    $$
    where
    $$
    M^\pm_2(z) = M^\pm_2(z;q) : = -\srvpm (z\omega + K) P_q^{\leq} (H_0 - z)^{-1}  (z\omega + K^*) \srvpm,
    $$

  $$
     R^\pm_2(z) = R^\pm_2(z;q) : = -\srvpm (z\omega + K) P_q^{>} (H_0 - z)^{-1}  (z\omega + K^*) \srvpm.
    $$

    \beprl{p11}
    Fix $q \in \Z_+$. For every $E \in (-\infty, \Lambda_{q+1})$ there exists $R^\pm_2(E) = R^\pm_2(E)^* \in \gS_2$ such that
     $$
   \lim_{\C_- \ni z \to E} \|R^\pm_2(z) - R^\pm_2(E)\|_2 = 0,
    $$
    $R^\pm_2(E)$ depends continuously in $\gS_2$ on $E$, and
    $$
    \|R^\pm_2(E)\|_2 = O\left(\left(E^2+1\right)\left(1+ |E|(\Lambda_{q+1}-E)^{-1}\right)\right), \quad E \in (-\infty, \Lambda_{q+1}).
    $$
    \epr
    \begin{proof}
    We have
    $$
    R^\pm_2(z)  = -\srvpm (z\omega + K)H_0^{-1/2} \left(P_q^>\left(I + z (H_0 - z)^{-1}\right)\right) H_0^{-1/2} (z\omega + K^*) \srvpm.
    $$
    Now the claims of the proposition follow from the facts that by \eqref{46} we have $V_\pm \in \gS_2$, the operators $\omega H_0^{-1/2} $, $K H_0^{-1/2} $, and $P_q^>$, are bounded,
    $$
    {\rm n}-\lim_{\delta \to 0} P_q^> (H_0 - E + i\delta))^{-1} = \sum_{j > q} p_j \otimes (\Hopa + \Lambda_j - E)^{-1},
    $$
   the operator $\sum_{j > q} p_j \otimes (\Hopa + \Lambda_j - E)^{-1}$ depends continuously in $\gB$ on $E \in (-\infty,\Lambda_{q+1})$, and
    $$
    \|\sum_{j > q} p_j \otimes (\Hopa + \Lambda_j - E)^{-1}\| = (\Lambda_{q+1} - E)^{-1}.
    $$
    \end{proof}
    Further,
    $$
     M^\pm_2(z;q) = \sum_{j \leq q} M^\pm_{2,j}(z)
     $$
     where
     $$
     M^\pm_{2,j}(z) : = -\srvpm (z\omega + K) P_j (H_0 - z)^{-1} (z\omega + K^*) \srvpm, \quad z \in \C_-, \quad j \in \Z_+.
     $$
     Let $\omega_4 \in C_0^\infty(\re; [0,1])$ be such a function that $\omega_4(\xpa) = 1$ if $\xpa \in \pipa({\rm supp}\,\omega)$. Then,
     $$
     M^\pm_{2,j}(z) : = -\srvpm (z\omega + K) P_j \, \left(p_j \otimes \left(\omega_4 (\Hopa + \Lambda_j - z)^{-1} \omega_4\right)\right) \, P_j (z\omega + K^*) \srvpm, \quad j \leq q.
     $$
     Define the operator
     \bel{m20}
     L_j(z) : = (z\omega - \Delta \omega + \omega_1 G_1 + \omega_2 G_2)P_j, \quad z \in \C, \quad j \in \Z_+,
     \ee
     so that $(z\omega + K)P_j  = L_j(z) + \omega_3 G_3 P_j$.
     Set
     $$
     \cR(z) = \omega_4 (\Hopa  - z)^{-1} \omega_4, \quad \widetilde{\cR}(z) = \omega_4 \pd (\Hopa  - z)^{-1} \omega_4, \quad z \in \C_-.
     $$
     Then we have
     $$
     M^\pm_{2,j}(z) =
     $$
     $$
     -\srvpm \left(L_j(z) \; p_j \otimes \cR(z-\Lambda_j) L_j(\bar{z})^* - \omega_3  \; p_j \otimes (\Ipa - (\Lambda_j - z) \cR(z-\Lambda_j)) \; \omega_3\right)\srvpm
      $$
      \bel{43}
      + \srvpm \left(L_j(z) \; p_j \otimes \widetilde{\cR}(z-\Lambda_j) \; \omega_3 + \omega_3  \; p_j \otimes \widetilde{\cR}(z-\Lambda_j) \; L_j(\bar{z})^* \right) \srvpm.
     \ee
     \belel{l3} Let $E \in \re \setminus \{0\}$. Then there exist operators $\cR(E) , \widetilde{\cR}(E) \in \gS_2(L^2(\re))$ such that
     $$
     {\rm n}-\lim_{\C_- \ni z \to E}\|\cR(z)-\cR(E)\|_2, \quad   {\rm n}-\lim_{\C_- \ni z \to E}\|\widetilde{\cR}(z) -\widetilde{\cR}(E)\|_2.
     $$
     Moreover, the operator $\cR(E)$ admits the integral kernel
     \bel{sep4}
     \cK_E(\xpa,\xpa') = \left\{
     \begin{array} {l}
     \frac{1}{2\sqrt{|E|}}w_4(\xpa) e^{-\sqrt{|E|}|\xpa-\xpa'|} w_4(\xpa'), \quad E<0,\\[4mm]
     -\frac{i}{2\sqrt{E}}w_4(\xpa) e^{-i\sqrt{E}|\xpa -\xpa'|} w_4(\xpa'), \quad E>0,
     \end{array}
     \right.
     \quad
     \xpa, \xpa' \in \re,
    \ee
     while the operator $\widetilde{\cR}(E)$ admits the integral kernel
     $$
     \widetilde{\cK}_E(\xpa,\xpa') = \left\{
     \begin{array} {l}
     \frac{-{\rm sign}\,(\xpa-\xpa')}{2}w_4(\xpa) e^{-\sqrt{|E|}|\xpa-\xpa'|} w_4(\xpa'), \quad E<0,\\[4mm]
    -\frac{{\rm sign}\,(\xpa-\xpa')}{2} w_4(\xpa) e^{-i\sqrt{E}|\xpa -\xpa'|} w_4(\xpa'), \quad E>0,
     \end{array}
     \right.
     \quad
     \xpa, \xpa' \in \re,
     $$
     so that $\cR(E)$ and $\widetilde{\cR}(E)$ depend continuously in $\gS_2(L^2(\re))$ on $E \in \re \setminus \{0\}$, and
     $$
      \|\cR(E)\|_2 \leq (2\sqrt{|E|})^{-1} \|\omega_4\|_{L^2(\re)}^2, \quad \|\widetilde{\cR}(E)\|_2 \leq 2^{-1} \|\omega_4\|_{L^2(\re)}^2, \quad E \in \re\setminus \left\{0\right\}.
      $$
      \ele
      We omit the  proof based on elementary facts from complex and functional analysis. \\

      {\em Remark}: In fact, $\cR(E) \in \gS_1(L^2(\re))$ (see \cite[Eq. (4.4)]{BrPuRa04}) but we will not use this in the article. \\

      For $j \in \Z_+$ and $E \in \re \setminus \{\Lambda_j\}$ set
      $$
       M^\pm_{2,j}(E) =
     $$
     $$
     -\srvpm \left(L_j(E) \; p_j \otimes \cR(E-\Lambda_j) L_j(E)^* - \omega_3 \; p_j \otimes (\Ipa - (\Lambda_j - E) \cR(E-\Lambda_j)) \; \omega_3\right)\srvpm
      $$
      \bel{48}
      + \srvpm \left(L_j(E) \; p_j \otimes \widetilde{\cR}(E-\Lambda_j) \; \omega_3 + \omega_3 P_j \; p_j \otimes \widetilde{\cR}(E-\Lambda_j) \; L_j(E)^* \right) \srvpm.
     \ee

      \beprl{p14}
      Let $j \in \Z_+$ and $E \in \re \setminus \{\Lambda_j\}$. Then we have $${\rm Re}\, M^\pm_{2,j}(E) \in \gS_2, \quad {\rm Im}\, M^\pm_{2,j}(E) \in \gS_1,$$
      $$
       \lim_{\C_- \ni z \to E}\|M^\pm_{2,j}(z) -  M^\pm_{2,j}(E)\|_2 = 0,
       $$
     the operator ${\rm Re}\,M^\pm_{2,j}(E)$ {\rm (}resp., ${\rm Im}\,M^\pm_{2,j}(E)${\rm )} depends continuously in $\gS_2$ {\rm (} resp., in $\gS_1${\rm )} on $E$,
     and
     $$
    \|{\rm Re}\,M^\pm_{2,j}(E)\|_2, \; \|{\rm Im}\,M^\pm_{2,j}(E)\|_1 = O\left(\left(E^2 + 1\right)|E-\Lambda_j|^{-1/2}\right), \quad E \in \re\setminus\{\Lambda_j\}.
     $$
     \epr
     \begin{proof}
     Set
     $$
     F_{1,j}(E) : = L_j(E) \; p_j \otimes \cR(E-\Lambda_j) L_j(E)^*, \quad
     F_{2,j}(E) : = - (\Lambda_j - E) \omega_3  \; p_j \otimes \cR(E-\Lambda_j) \; \omega_3,
     $$
      $$
     F_{3,j}(E) : = - L_j(E) \; p_j \otimes \widetilde{\cR}(E-\Lambda_j) \; \omega_3, \quad
     F_{4,j}(E) : = - \omega_3 P_j \; p_j \otimes \widetilde{\cR}(E-\Lambda_j) \; L_j(E)^* .
     $$
     Then,
     \bel{au10}
     M^\pm_{2,j}(E)  = - \sum_{\ell=1}^4 \srvpm F_{\ell, j}(E) \srvpm - \srvpm \omega_3 P_j \omega_3 \srvpm
     \ee
     so that
     $$
     {\rm Re}\, M^\pm_{2,j}(E)  = - \sum_{\ell=1}^4 {\rm Re}\,(\srvpm F_{\ell, j}(E) \srvpm) - \srvpm \omega_3 P_j \omega_3 \srvpm,
       $$
       $$
       {\rm Im}\,M^\pm_{2,j}(E)  = - \sum_{\ell=1}^4 {\rm Im}\,(\srvpm F_{\ell, j}(E) \srvpm).
     $$
      Taking into account Lemma \ref{l3} and the facts that the orthogonal projection $p_j$ has an integral kernel  in $C^\infty(\rd \times \rd)$ while the functions $\omega$ and $\omega_k$, $k=1,2,3$, are in $C_0^\infty(\re^3)$, we find that $F_{\ell,j}$, $\ell = 1,\ldots,4$,  are continuous functions from $\re\setminus \{\Lambda_j\}$ to $\gS_2$, and
      $$
      \|F_{1,j}(E)\|_2 = O\left(\left(E^2 + 1\right)|E-\Lambda_j|^{-1/2}\right),
      $$
      \bel{sep2}
      \|F_{2,j}(E)\|_2 = O\left(\left(|E| + 1\right)|E-\Lambda_j|^{1/2}\right),
      \ee
     \bel{sep3}
      \|F_{\ell,j}(E)\|_2 = O\left(|E| + 1\right), \quad \ell = 3,4.
      \ee
      Since, by \eqref{46}, we have $V_\pm \in \gS_2$, we find that $\srvpm F_{\ell, j}(E) \srvpm \in \gS_1$. Moreover, the continuity of $F_{\ell, j}$ in $\gS_2$ implies the continuity of $\srvpm F_{\ell, j}(E) \srvpm$ in $\gS_1$, and
      $$
      \|\srvpm F_{\ell, j}(E) \srvpm\|_1 \leq \|V_\pm\|_2 \, \|F_{\ell, j}\|_2, \quad \ell =1,\ldots,4.
      $$
      Finally, by \eqref{46}, we have
      $\srvpm \omega_3 P_j \omega_3 \srvpm \in \gS_2$. Therefore, the claims of the proposition follow from  representation \eqref{au10} and the properties of  $F_{\ell, j}$ established above.

     \end{proof}

Now Proposition \ref{p8} follows from Propositions \ref{p11}, \ref{p13}, and \ref{p14}.


\subsection{Proof of Proposition \ref{fp1}}
\label{fs2}
As above, we denote by $X$ a separable Hilbert space.
\begin{lemma} \label{lau2} {\rm \cite[Lemma 2.1]{Pu97}}
Let $T_1 = T_1^* \in \gS_\infty(X)$,  $T_2 = T_2^* \in \gS_1(X)$. Then for any $s>0$ we have
$$
\frac{1}{\pi} \int_\re n_\pm(s; T_1 + t T_2) \frac{dt}{1+t^2}\leq n_\pm(s/2, T_1) + \frac{2}{\pi s} \|T_2\|_1.
$$
\end{lemma}
Our next lemma contains an elementary Chebyshev-type estimate for the eigenvalue counting functions of compact operators.
\begin{lemma} \label{lau3}
Let $T= T^* \in \gS_p(X)$, $p \in [1,\infty)$. Then for any $s>0$ we have
$$
n_*(s; T) : = n_+(s; T) + n_-(s; T) \leq s^{-p} \|T\|^p_p.
$$
\end{lemma}
By Lemma \ref{lau2} with $s=1$, and Lemma \ref{lau3} with $s=1/2$ and $p=2$, we obtain
    \bel{au60}
    |\tilde{\xi}(E;H_\pm, H_0)| \leq 4\|{\rm Re}\,T^\pm(E)\|_2^2 + \frac{2}{\pi} \|{\rm Im}\,T^\pm(E)\|_1.
    \ee
    Combining \eqref{au60} with Proposition \ref{p8}, we find that $\tilde{\xi}(E;H_\pm, H_0)$ is well defined for any $E \in (0,\infty)\setminus b(2\Z_++1)$, and $\tilde{\xi}(\cdot;H_\pm, H_0)$ is bounded on every compact subset of
 $(0,\infty)\setminus b(2\Z_++1)$. \\
    Let us now prove the continuity of $\tilde{\xi}(\cdot;H_\pm, H_0)$ following the main ideas of the proof of the continuity part of \cite[Proposition 2.5]{BrPuRa04}. Let $E_0 \in (0,\infty)\setminus b(2\Z_++1)$. Assume that
    \bel{au61}
    \lim_{E \to E_0} \|{\rm Re}\,T^\pm(E) - {\rm Re}\,T^\pm(E_0)\| =  \lim_{E \to E_0} \|{\rm Im}\,T^\pm(E) - {\rm Re}\,T^\pm(E_0)\|_1 = 0,
    \ee
    \bel{au62}
    \pm 1 \not \in \sigma(T^\pm(E_0)).
    \ee
    Then, by  \cite[Lemma 2.5]{Pu97} we have
    $$
     \lim_{E \to E_0} \tilde{\xi}(E;H_\pm, H_0) = \tilde{\xi}(E_0;H_\pm, H_0).
     $$
     Proposition \ref{p8} implies \eqref{au61} for any $E_0 \in (0,\infty)\setminus b(2\Z_++1)$. Moreover, \eqref{au62} with $E_0 \in (0,\infty)\setminus (\sigma_{\rm p}(H_\pm) b(2\Z_++1))$ will follow from
     \begin{lemma} \label{lau4}
     Let $E \in (0,\infty)\setminus b(2\Z_++1)$. Assume that
     \bel{au98} \pm 1 \in \sigma_{\rm p}(T^\pm(E)). \ee
     Then
     \bel{au99}
     E \in \sigma_{\rm p}(H_\pm).
     \ee
     \end{lemma}
     \begin{proof}
     If $E \in (0,b)$ then $E \in \rho(H_0) : = \C\setminus \sigma(H_0)$ so that in the Neumann case the lemma follows from the Birman-Schwinger principle. In the Dirichlet case, \eqref{au99} and, hence, \eqref{au98} cannot hold true with $E \in (0,b)$.
    That is why we assume that $E>b$, and will follow the general lines of the proof of \cite[Section XIII.8, Lemma 8]{ReSiIV}.\\ Let $0 \neq \varphi = \varphi^\pm \in L^2(\re^3)$ satisfy
     \bel{au70}
     T^\pm(E)\varphi = \pm \varphi.
     \ee
     Set $\phi : = \srvpm \varphi$, and
     $$
     w_s(t) : = (1+ t^2)^{s/2}, \quad t \in \re, \quad s \in \re.
     $$
     As usual, we denote the multiplier by $w_s$ acting in $L^2(\re)$ by the same symbol $w_s$. Moreover, for $s \in \re$, set
     $$
     (W_{s} u)(\xpe,\xpa) : = w_s(\xpa) u(\xpe, \xpa), \quad (\xpe, \xpa) \in \re^3, \quad u \in \gD(W_{s}) \subset L^2(\re^3).
     $$
     Writing $\phi = H_0^{-1} \omega H_0 \phi$ with $\omega \in C_0^\infty(\re^3; [0,1])$ such that $\omega = 1$ in a neighborhood of $\Gamma$ (see Lemma \ref{l1}),  and commuting $W_{s}$ with $H_0^{-1}$ appropriately many times, we easily find that
     \bel{au100}
     W_{s} \phi \in L^2(\re^3), \quad s \in \re.
     \ee
     Let $\left\{\varphi_{k,q}\right\}_{k \in \Z_+}$ be an orthogonal basis of ${\rm Ran}\,p_q$, $q \in \Z_+$, for example the canonic basis defined in \eqref{au73} -- \eqref{au74} below. Set
     $$
     \phi_{k,q}(\xpa) : = \int_{\rd} \phi(\xpe,\xpa) \overline{\varphi_{k,q}(\xpe)}\, d\xpe, \quad \xpa \in \re, \quad k,q \in \Z_+.
     $$
     Evidently,
     $$
     \|\phi\|_{L^2(\re^3)}^2 = \sum_{(k,q) \in \Z_+^2} \, \|\phi_{k,q}\|^2_{L^2(\re)}.
     $$
     Moreover, by \eqref{au100}, we find that
     \bel{au101}
     w_s \phi_{k,q} \in L^2(\re), \quad s \in \re, \quad k,q \in \Z_+.
     \ee
     Further, for any $z \in \rho(H_0)$ we have
     \bel{au75}
     (H_0-z)^{-1} = \sum_{q \in \Z_+} p_q \otimes (\Hopa + \Lambda_q - z)^{-1}.
     \ee Let $q_0$ be the largest integer satisfying $q_0 < \frac{E-b}{2b}$. Since $E>b$, we have $q_0 \geq 0$.\\
     By \eqref{au70},  \eqref{8}, and \eqref{au75}, we get
     $$
     0 = \lim_{\varepsilon \downarrow 0} {\rm Im}\,\langle(H_0^{-1}-(E+i\varepsilon)^{-1})^{-1}\phi, \phi\rangle = -E^2 \lim_{\varepsilon \downarrow 0} {\rm Im}\,\langle(H_0-E-i\varepsilon)^{-1}\phi, \phi\rangle =
     $$
     $$
    - \frac{\pi E^2}{2} \sum_{q=0}^{q_0} \sum_{k \in \Z_+} (E-\Lambda_q)^{-1/2} \left(\left|\widehat{\phi_{k,q}}\left(-\sqrt{E-\Lambda_q}\right)\right|^2 + \left|\widehat{\phi_{k,q}}\left(\sqrt{E-\Lambda_q}\right)\right|^2\right),
    $$
    that is
    \bel{au102}
    \widehat{\phi_{k,q}}\left(\pm\sqrt{E-\Lambda_q}\right) = 0, \quad k \in \Z_+, \quad, q=0,\ldots,q_0,
    \ee
    where $\widehat{\phi_{k,q}}$ is the  Fourier transform of $\phi_{k,q}$. Then, by \cite[Section IX.9, Lemma 3,]{ReSiII}, relations  \eqref{au101} and \eqref{au102} imply the existence of a function $\beta_{k,q} \in L^2(\re)$ such that
    $$
    \beta_{k,q} = \lim_{\varepsilon \downarrow 0} \, (\Hopa + \Lambda_q - E - i\varepsilon)^{-1} \phi_{k,q}, \quad k \in \Z_+, \quad q=0,\ldots,q_0.
    $$
    Set
     $$
    \beta_{k,q} : =  (\Hopa + \Lambda_q - E)^{-1} \phi_{k,q}, \quad k \in \Z_+, \quad q > q_0,
    $$
    and
    $$
    \beta(\xpe,\xpa) : = \sum_{(k,q) \in \Z_+^2} \beta_{k,q}(\xpa)\,\varphi_{k,q}(\xpe), \quad (\xpe,\xpa) \in \re^3.
    $$
    Then $\beta \in L^2(\re^3)$, and
   \bel{au76}
    \beta = \lim_{\varepsilon \downarrow 0} \, (H_0  - E - i\varepsilon)^{-1}\,\phi.
    \ee
    Set $\psi : = - E^2 \beta - E \phi$. Then, by \eqref{8} and \eqref{au76}, we have
    \bel{au77}
    \psi = \lim_{\varepsilon \downarrow 0} \, (H_0^{-1}  - (E + i\varepsilon)^{-1})^{-1}\,\phi \in L^2(\re^3).
    \ee
    Moreover, \eqref{au70} implies
      \bel{au78}
      V_\pm \psi = \pm \phi.
      \ee
      By \eqref{au77} and \eqref{au78}, we easily find that
      $$
      H_0^{-1} \psi = \pm V_\pm \psi + E^{-1} \psi
      $$
      which is equivalent to $H_\pm \psi = E \psi$, $\psi \in \gD(H_\pm)$. Since $\psi \neq 0$, we arrive at \eqref{au99}.
     \end{proof}
     Finally, we prove \eqref{39a}, following the general ideas of \cite{Pu97}. By the invariance principle,
     \bel{au93}
     \xi(E; H_\pm) = - \xi(E^{-1}; H_\pm^{-1}, H_0^{-1}), \quad E \in (0,\infty),
     \ee
     where $\xi(E^{-1}; H_\pm^{-1}, H_0^{-1}) : = \xi(E^{-2}; H_\pm^{-2}, H_0^{-2})$, and $\xi(E^{-2}; H_\pm^{-2}, H_0^{-2})$ is defined by
     \eqref{37a}.\\
     Let $\left\{\lambda_j^\pm\right\}_{j \in \N}$ be the non-increasing sequence of the non-zero eigenvalues of $V_\pm$, and $\left\{f_j^\pm\right\}_{j \in \N}$ be the corresponding orthonormal eigenfunctions, so that
     $$
     V_\pm = \sum_{j \in \N} \lambda_j^\pm \langle \cdot, f_j^\pm\rangle f_j^\pm.
     $$
     For $\ell \in \N$ set
     $$
     V_{\pm,\ell} : = \sum_{j = 1}^\ell \lambda_j^\pm \langle \cdot, f_j^\pm\rangle f_j^\pm, \quad
     S_{\pm, \ell} : = H_0^{-1} \mp V_{\pm,\ell},
     $$
     $$
     T^\pm_\ell(E) : = {\rm n}-\lim_{\C_- \ni z \to E}  V_{\pm,\ell}^{1/2} \, (H_0^{-1}  - z^{-1})^{-1} \, V_{\pm_,\ell}^{1/2}, \quad E \in (0,\infty)\setminus b(2\Z_+ + 1).
     $$
     It is easy to check that ${\rm Re}\,T^\pm_\ell(E) \in \gS_2$, ${\rm Im}\,T^\pm_\ell(E) \in \gS_1$, and
     \bel{au92}
     \lim_{\ell \to \infty}\|{\rm Re}\,T^\pm_\ell(E) - {\rm Re}\,T^\pm(E)\|_2 = \lim_{\ell \to \infty}\|{\rm Im}\,T^\pm_\ell(E) - {\rm Im}\,T^\pm(E)\|_1 = 0.
     \ee
     Since the ranks of the operators $V_{\pm,\ell}$, $\ell \in \N$ are finite, and, hence $V_{\pm,\ell} \in \gS_1$, we find that \cite[Theorem 1.1]{Pu97} implies
     \bel{au91}
     \xi(E^{-1}; S_{\pm, \ell}, H_0^{-1}) = \mp \frac{1}{\pi} \int_\re n_\pm(1; {\rm Re}\,T^\pm_\ell(E) + t {\rm Im}\,T^\pm_\ell(E)) \frac{dt}{1+t^2}
     \ee
     for almost every $E \in (0,\infty)$. It remains to pass to the limit as $\ell \to \infty$ at both hand sides of \eqref{au91}.
    We have
      \bel{au96}
      \xi(E^{-1}; S_{\pm, \ell}, H_0^{-1}) = \xi(E^{-1}; S_{\pm, \ell}, H_\pm^{-1}) + \xi(E^{-1}; H_\pm^{-1}, H_0^{-1}).
      \ee
     Bearing in mind  \eqref{47}, we apply \cite[Lemma 4.2]{Pu97}, and obtain
      \bel{au97}
      \lim_{\ell \to \infty} \xi(E^{-1}; S_{\pm, \ell}, H_\pm^{-1}) = 0
      \ee
      for almost every $E \in (0,\infty)$.\\
     Next, combining \eqref{au91}, \eqref{au92}, and Lemma \ref{lau4}, we find that \cite[Lemma 2.5]{Pu97} implies
     \bel{au95}
     \lim_{\ell \to \infty} \int_\re n_\pm(1; {\rm Re}\,T^\pm_\ell(E) + t {\rm Im}\,T^\pm_\ell(E)) \frac{dt}{1+t^2} = \pm \tilde{\xi}(E; H_\pm, H_0),
     \ee
     for every $E \in (0,\infty)\setminus (\sigma_{\rm p}(H_\pm) \cup b(2\Z_++1))$.\\
     Putting together \eqref{au93}, \eqref{au91}, and \eqref{au96}-\eqref{au95}, we obtain \eqref{39a}.

 \section{Proof of Theorem \ref{th1}}
\label{s5}
Throughout the section the parameter $q \in \Z_+$ is fixed as in Theorem \ref{th1}.
    \subsection{The effective Hamiltonians}
    \label{ss1}

Define the rank-one operator
$$
\gp : = \langle \cdot, \omega_4\rangle \omega_4,
$$
self-adjoint in $L^2(\re)$.
For  $\lambda \in (-b,0) \cup (0,b)$, set
$$
M^\pm_{3,q}(\lambda) : = - \frac{\iota(\lambda)}{2\sqrt{|\lambda|}} \cM^\pm_q
$$
where
$$
\iota(\lambda) : = \left\{
\begin{array} {l}
1 \quad {\rm if} \quad \lambda <0, \\
-i \quad {\rm if} \quad \lambda >0,
\end{array}
\right.
$$
$$
\cM^\pm_q : = \srvpm \cT_q p_q \otimes \gp \cT_q^* \srvpm,
$$
and
    \bel{m21}
\cT_q : = L_q(\Lambda_q),
    \ee
    the operator $L_q(z)$ being defined in \eqref{m20}. Note that the operators $\cM^\pm_q$ are self-adjoint and non-negative so that the operators $M^\pm_{3,q}(\lambda)$ are self-adjoint and non-positive if
    $\lambda < 0$ and purely imaginary if $\lambda > 0$.

\beprl{p2} Let $q \in \Z_+$, $\epsilon \in (0,1)$. Then we have
    \bel{f6}
\xi(\Lambda_q-\lambda; H_+, H_0) = O(1),
    \ee
    \begin{align} \label{f7}
-n_+((1-\epsilon)2\sqrt{\lambda};\cM^-_q) + O(1) &\leq \xi(\Lambda_q-\lambda; H_-, H_0)\\[2mm] \nonumber
& \leq -n_+((1+\epsilon)2\sqrt{\lambda};\cM^-_q) + O(1),
    \end{align}
    \begin{align} \label{f8}
\frac{1}{\pi} {\rm Tr}\,\arctan{\left(\frac{\cM^+_q}
{(1+\epsilon)2\sqrt{\lambda}}\right)}  + O(1) &\leq \xi(\Lambda_q+\lambda; H_+, H_0)\\[3mm] \nonumber
 & \leq \frac{1}{\pi} {\rm Tr}\,\arctan{\left(\frac{\cM^+_q}{(1-\epsilon)2\sqrt{\lambda}}\right)}  + O(1),
    \end{align}

    \begin{align} \label{f9}
-\frac{1}{\pi} {\rm Tr}\, \arctan{\left(\frac{\cM^-_q}{(1-\epsilon)2\sqrt{\lambda}}\right)}  + O(1) &\leq  \xi(\Lambda_q+\lambda; H_-, H_0)\\[3mm] \nonumber
 &\leq -\frac{1}{\pi} {\rm Tr}\,\arctan{\left(\frac{\cM^-_q}{(1+\epsilon)2\sqrt{\lambda}}\right)}  + O(1),
    \end{align}
 as $\lambda \downarrow 0$.
 \epr
 {\em Remark}: According to Proposition \ref{p2}, the operators $\cM_q^\pm$ play the role of {\em effective Hamiltonians} in the asymptotic analysis of the
SSF $\xi(E;H\pm, H_0)$ as the energy $E$ approaches the Landau level $\Lambda_q$, $q \in \Z_+$. \\

For the proof of Proposition \ref{p2} we need the well known {\em Weyl inequalities} for the eigenvalues of compact operators, described in the following
\begin{lemma} \label{lau1} {\rm \cite[Theorem 9, Section 9.2]{BiSo87}}
Let $X$ be a separable Hilbert space, and $T_j = T_j^* \in \gS_\infty(X)$, $j=1,2$. Then for any $s_j>0$ we have
$$
n_\pm(s_1+s_2; T_1 + T_2) \leq n_\pm(s_1, T_1) + n_\pm(s_2, T_2).
$$
\end{lemma}

 \begin{proof}[Proof of Proposition \ref{p2}]
 Set
$$
R^\pm_{3,q}(\lambda) : = T^\pm(\Lambda_q + \lambda) - M^\pm_{3,q}(\lambda).
$$
By Propositions \ref{p11}, \ref{p13}, and \ref{p14}, estimates \eqref{sep2}-\eqref{sep3}, and the explicit form \eqref{sep4} of the integral kernel of the operator $\cR(\lambda)$, we have
    \bel{au11}
\|{\rm Re}\;R^\pm_{3,q}(\lambda)\|_2 = O(1), \quad \|{\rm Im}\;R^\pm_{3,q}(\lambda)\|_1 = O(1), \quad \lambda \downarrow 0.
    \ee
Applying Lemma \ref{lau1}, we get
$$
\frac{1}{\pi}\int_\re n_\pm(1+\epsilon; {\rm Re}\;M^\pm_{3,q}(\lambda) + t {\rm Im}\;M^\pm_{3,q}(\lambda)) \frac{dt}{1+t^2}
\, -  $$
$$
\frac{1}{\pi} \int_\re n_* (\epsilon; {\rm Re}\;R^\pm_{3,q}(\lambda) + t {\rm Im}\;R^\pm_{3,q}(\lambda)) \frac{dt}{1+t^2} \leq
$$
$$
\frac{1}{\pi}\int_\re n_\pm(1; {\rm Re}\;T^\pm(\Lambda_q +\lambda) + t {\rm Im}\;T^\pm(\Lambda_q +\lambda)) \frac{dt}{1+t^2}  \leq
$$
$$
\frac{1}{\pi}\int_\re n_\pm(1-\epsilon; {\rm Re}\;M^\pm_{3,q}(\lambda) + t {\rm Im}\;M^\pm_{3,q}(\lambda)) \frac{dt}{1+t^2} \,
+ $$
    \bel{au12}
\frac{1}{\pi} \int_\re n_*(\epsilon; {\rm Re}\;R^\pm_{3,q}(\lambda) + t {\rm Im}\;R^\pm_{3,q}(\lambda)) \frac{dt}{1+t^2}.
    \ee
    Using Lemmas \ref{lau2}, and \ref{lau3} with $s>0$ and $p=2$, we obtain
    \bel{au13}
    \frac{1}{\pi}\int_\re n_*(s; {\rm Re}\;R^\pm_{3,q}(\lambda) + t {\rm Im}\;R^\pm_{3,q}(\lambda)) \frac{dt}{1+t^2} \leq \frac{4}{s^2} \|{\rm Re}\;R^\pm_{3,q}(\lambda)\|^2_2 + \frac{2}{\pi s} \|{\rm Im}\;R^\pm_{3,q}(\lambda)\|_1.
    \ee
    Putting together \eqref{39}, \eqref{au12}, \eqref{au13}, and \eqref{au11}, we get
     $$
\frac{1}{\pi} \int_\re n_\pm (1+\epsilon; {\rm Re}\;M^\pm_{3,q}(\lambda) + t {\rm Im}\;M^\pm_{3,q}(\lambda)) \frac{dt}{1+t^2} + O(1) \leq
$$
$$
\pm \xi(\Lambda_q + \lambda; H_\pm, H_0) \leq
$$
    \bel{au15}
    \frac{1}{\pi}\int_\re n_\pm(1-\epsilon; {\rm Re}\;M^\pm_{3,q}(\lambda) + t {\rm Im}\;M^\pm_{3,q}(\lambda)) \frac{dt}{1+t^2} \,
+ O(1)
\ee
as $\lambda \to 0$. Simple calculations show that for $s>0$ we have
    \bel{au16}
    \frac{1}{\pi}\int_\re n_+(s; {\rm Re}\;M^\pm_{3,q}(\lambda) + t {\rm Im}\;M^\pm_{3,q}(\lambda)) \frac{dt}{1+t^2} = 0,
    \ee
    \bel{au17}
    \frac{1}{\pi}\int_\re n_-(s; {\rm Re}\;M^\pm_{3,q}(\lambda) + t {\rm Im}\;M^\pm_{3,q}(\lambda)) \frac{dt}{1+t^2} = n_+(2s\sqrt{|\lambda|}; \cM^\pm_q),
    \ee if
    $\lambda < 0$,
    and
    \bel{au18}
    \frac{1}{\pi}\int_\re n_\pm(s; {\rm Re}\;M^\pm_{3,q}(\lambda) + t {\rm Im}\;M^\pm_{3,q}(\lambda)) \frac{dt}{1+t^2} =
    \frac{1}{\pi} {\rm Tr}\,\arctan{\left(\frac{\cM^+_q}{2s\sqrt{\lambda}}\right)},
    \ee
    if
    $\lambda > 0$. Now the claims of the proposition follow from estimates \eqref{au15} and identities \eqref{au16} - \eqref{au18}.
 \end{proof}
 Note that \eqref{f6} is identical with \eqref{ms1}, so that in order to complete the proof of Theorem \ref{th1}, it  remains to prove \eqref{m13} - \eqref{ms2} using respectively \eqref{f7} and \eqref{f8} -- \eqref{f9}.
Here we state two more lemmas needed for the estimates of $n_+(s; \cM_q^\pm)$.

    \begin{lemma} \label{lm1}
    Let $X_j$, $j=1,2$, be Hilbert spaces, and $J: X_1 \to X_2$ be a linear compact operator. Then we have
    $$
    n_+(s; J^* J) = n_+(s; J J^*), \quad s>0.
    $$
    \end{lemma}
    \begin{proof}
    The claim follows immediately form \cite[Chapter 8, Section 1, Theorem 4]{BiSo87}.
    \end{proof}
    \begin{lemma} \label{lm2}
    Let $X_j$, $j=1,2$, be Hilbert spaces, $J: X_1 \to X_2$ be a linear compact operator, and $T \in \gS_\infty(X_2)$. Then we have
    \bel{m1}
   n_+(s; J^*(I-T) J) \geq n_+(s; (1-\veps)J^*J) - {\rm Tr}\,\one_{[\veps,\infty)}(T), \quad s>0, \quad \veps \in (0,1).
    \ee
    \end{lemma}
    \begin{proof}
    We have
    \bel{m1a}
    J^*(I-T) J = J^*\,((1-\veps)I + \one_{(-\infty, \veps)}(T) (\veps I - T) +  \one_{[\veps, \infty)}(T) (\veps I - T))J.
    \ee
    Evidently,
    \bel{m1b}
    \one_{(-\infty, \veps)}(T) (\veps I - T)  \geq 0, \quad {\rm rank} \, J^* \one_{[\veps, \infty)}(T) (\veps I - T)J \leq {\rm Tr}\,\one_{[\veps,\infty)}(T).
    \ee
   By the mini-max principle and \cite[Chapter 9, Section 3, Theorem 3]{BiSo87}, now \eqref{m1} follows from \eqref{m1a} and \eqref{m1b}.
    \end{proof}

    For further references, set
    \bel{m3}
    M^\pm_{4,q} : = (p_q \otimes \gp) \cT_q^* V_\pm \cT_q (p_q \otimes \gp).
    \ee
    The operator $M^\pm_{4,q}$ will be considered as a compact self-adjoint operator in the Hilbert space $(p_q \otimes \gp) L^2(\re^3)$.
    By Lemma \ref{lm1}, we have
    \bel{m4}
    n_+(s; \cM^\pm_q) = n_+(s;M^\pm_{4,q}), \quad s>0.
    \ee

 \subsection{Lower bounds of $n_+(s; M_{4,q}^+)$ in the Dirichlet case}
       \label{ss3}

In this and in the following subsection we assume $\omega = 1$ in a neighborhood of $\overline{\omin}$, where $\omega$ is the function which participates in the definition of the operator $L_q(z)$ (see \eqref{m20}), and hence in that of $\cT_q$ (see \eqref{m21}), and of $M^\pm_{4,q}$ (see \eqref{m3}).\\ Let $\Omega \subset \re^3$ be a bounded domain. Note that for any $\xpe \in \re$ the function
$$
\re \ni \xpa \mapsto \one_{\Omega}(\xpe, \xpa) \in \{0,1\}
$$
is Lebesgue measurable and has a bounded support. Set
    \bel{fin60}
    w_{\Omega}(\xpe) : = \int_\re \one_{\Omega}(\xpe, \xpa) \, d\xpa, \quad \xpe \in \rd.
    \ee
    Evidently, $w_{\Omega}(\xpe) \geq 0$ for every $\xpe \in \rd$, and $w_{\Omega}(\xpe)>0$ if and only if $\xpe \in \pipe(\Omega)$.

\beprl{Sy1}
Let the domain $\oless \subset \re^3$ satisfy $\overline{\oless} \subset \omin$. Then we have
    \bel{m22}
n_+(s; M^+_{4,q}) \geq n_+\left(4s\|\omega_4\|_{L^2(\re)}^2;  p_q  \,w_{\Omega_<}\,  p_q\right)  + O(1), \quad s > 0.
    \ee
\epr
\bepf
By definition of $\cT_q$ (see \eqref{m21}), we have
$$
\cT_q = L_q(\Lambda_q)= (\Lambda_q \omega +K) P_q - \omega_3 P_q G_3= (\Lambda_q \omega - \omega H_0) P_q + H_0  \omega P_q - \omega_3 P_q G_3.
$$
Using that $\partial_{x_3} \omega_4 = 0$ on the support of $\omega$ and hence of $\omega_3$, we obtain
$$
\omega_3  G_3 (p_q  \otimes  \gp)=0, \quad \omega H_0 P_q  (p_q  \otimes  \gp)= \Lambda_q \omega P_q  (p_q  \otimes  \gp), \quad \cT_q (p_q  \otimes  \gp) = H_0 \omega P_q  (p_q  \otimes  \gp),
$$
 and, hence,
    \bel{m29}
 M^+_{4,q} = (p_q  \otimes  \gp)  P_q \omega  H_0  V_+ H_0 \omega P_q (p_q  \otimes \gp).
    \ee
On the other hand,
$$
V_+  = H_0^{-1} - H^{-1}_{+, {\rm in}}  \oplus  H^{-1}_{+, {\rm ex}}= V_{+,0} - R_{\rm in},
$$
with
$$
V_{+,0} : = H_0^{-1} - 0 \oplus H^{-1}_{+, {\rm ex}}, \qquad R_{\rm in} := H^{-1}_{+, {\rm in}} \oplus 0.
$$
Obviously,
$$
-R^{\frac12}_{\rm in} = \Lambda_q R^{\frac12}_{\rm in}  \Big( 0 \oplus H^{-1}_{+, {\rm ex}} - \Lambda_q^{-1} \Big) = \Big( 0 \oplus H^{-1}_{+, {\rm ex}} - \Lambda_q^{-1} \Big) \Lambda_q R^{\frac12}_{\rm in},
$$
and since $0 \oplus H^{-1}_{+, {\rm ex}} = H_0^{-1} -V_{+,0}$, we have
$$
R_{\rm in} = \Lambda_q^2 \Big( H_0^{-1} - \Lambda_q^{-1} - V_{+,0}  \Big)  \, R_{\rm in} \, \Big( H_0^{-1} - \Lambda_q^{-1} - V_{+,0} \Big).
$$
Moreover, from the above relations and the fact that $\omega$ is equal to $1$ on $\overline{\omin}$, we obtain
$$
R_{\rm in} (H_0^{-1} - \Lambda_q^{-1}) H_0 \omega P_q (p_q  \otimes  \gp)=  R_{\rm in} (\omega -  \Lambda_q^{-1} \omega H_0) P_q (p_q  \otimes \gp)+  R_{\rm in} [\omega,H_0] P_q (p_q  \otimes  \gp) =0.
$$
Using this relation and the dual one, we deduce
$$
M^+_{4,q} = (p_q  \otimes \gp)  P_q \omega H_0  (V_{+,0} - R_{\rm in})  H_0 \omega P_q (p_q  \otimes  \gp)
$$
$$
= (p_q \otimes \gp) P_q \omega  H_0  ( V_{+,0} - \Lambda_q^2 V_{+,0}  R_{\rm in} V_{+,0})  H_0 \omega P_q (p_q \otimes \gp) .
$$
Since the operator $V^{\frac12}_{+,0} R_{\rm in} V_{+,0}^{\frac12} $ is compact, Lemma \ref{lm2} implies
   $$
n_+(s; M^+_{4,q}) \geq
 $$
     \bel{m23}
 n_+\left(2s; (p_q \otimes \gp) P_q \omega  H_0   V_{+,0} H_0 \omega P_q (p_q \otimes \gp)\right) - {\rm Tr}\;\one_{[1/2,\infty)}(V^{\frac12}_{+,0} R_{\rm in} V_{+,0}^{\frac12}), \quad s>0.
    \ee
Further,
    \bel{m24}
V_{+,0} =
 \Big( H_0^{-1} - ( H_0+\oneoless)^{-1} \Big) + \Big(  ( H_0+\oneoless)^{-1} - 0 \oplus H^{-1}_{+, {\rm ex}} \Big).
    \ee
By $\oless \cap \omex= \emptyset $, the restriction to $\gD(H^{\frac12}_{+, {\rm ex}})$ of the quadratic form of $H_0+\oneoless$ coincides with the one of $H_{+, {\rm ex}}$. Hence, by \cite[Proposition 2.1 (i)]{PuRo07}, the second term on the r.h.s. of \eqref{m24} is  non-negative. Next, the resolvent identity yields
$$
 H_0^{-1} - ( H_0+ \oneoless)^{-1} =  H_0^{-1} \oneoless  \Big( I - \oneoless ( H_0+\oneoless)^{-1}\oneoless  \Big) \oneoless H_0^{-1}.
$$
Thus, the mini-max principle implies
    \bel{m25}
    n_+\left(s; (p_q \otimes \gp) P_q \omega  H_0   V_{+,0} H_0 \omega P_q (p_q \otimes \gp)\right) \geq
    \ee
    $$
    n_+\left(s; (p_q \otimes \gp) P_q \omega  \oneoless  \Big( I - \oneoless ( H_0+\oneoless)^{-1}\oneoless  \Big) \oneoless  \omega P_q (p_q \otimes \gp)\right), \quad s>0.
    $$
    Applying Lemma \ref{lm2}, and taking into account that
    \bel{m27a}
    (p_q \otimes \gp) P_q \omega   \oneoless  \omega P_q (p_q \otimes \gp) = (p_q \otimes \gp) \oneoless  (p_q \otimes \gp),
    \ee
    we obtain
    \bel{m26}
    n_+\left(s; (p_q \otimes \gp) P_q \omega  \oneoless  \Big( I - \oneoless ( H_0+\oneoless)^{-1}\oneoless  \Big) \oneoless    \omega P_q (p_q \otimes \gp)\right) \geq
    \ee
    $$
    n_+\left(s; (p_q \otimes \gp)   \oneoless   (p_q \otimes \gp)\right)
    - {\rm Tr}\,\one_{[1/2, \infty)} \left(\oneoless  (H_0+\oneoless)^{-1} \oneoless\right),
     \quad s>0.
    $$
    Finally, the operator $(p_q \otimes \gp)   \oneoless   (p_q \otimes \gp)$ with domain $(p_q \otimes \gp) L^2(\re^3)$ is unitarily equivalent to the operator
    $\|\omega_4\|^{-2}_{L^2(\re)}p_q \, w_{\Omega_<} \, p_q$ with domain $p_q L^2(\rd)$, where $w_{\Omega_<}$ is the function defined in \eqref{fin60}. Therefore,
    \bel{m27}
    n_+\left(s; (p_q \otimes \gp)   \oneoless   (p_q \otimes \gp)\right) = n_+\left(\|\omega_4\|^{2}_{L^2(\re)}s;  p_q w_{\Omega_<} p_q)\right), \quad s>0.
    \ee
    Now \eqref{m22} follows from \eqref{m23}, \eqref{m25}, \eqref{m26}, and \eqref{m27}.
    \epf

\subsection{Lower bounds of $n_+(s; M^-_{4,q})$ in the Neumann case}
\label{ss4}


\beprl{Sy2}Let the domain $\oless \subset \re^3$ satisfy $\overline{\oless} \subset \omin$. Then there exists a constant $c>0$ such that
    \bel{m30}
n_+(s; M^-_{4,q}) \geq n_+\left(cs\|\omega_4\|_{L^2(\re)}^2;  p_q  \,w_{\Omega_<}\,  p_q\right)  + O(1), \quad s > 0,
\ee
 \epr

\bepf
By analogy with \eqref{m29}, we obtain
    $$
M_{4,q}^- = (p_q \otimes \gp)  P_q \omega  H_0  V_- H_0 \omega P_q (p_q \otimes \gp).
    $$
On the other hand, for any $\delta >0$, we have
$$V_-: = H^{-1}_{-, {\rm in}} \oplus H^{-1}_{-, {\rm ex}} - H_0^{-1}  = V_{-,0} - R_{\rm in},$$
with
\bel{defv0}
V_{-,0} := (1 + \delta)  H^{-1}_{-, {\rm in}} \oplus H^{-1}_{-, {\rm ex}} - H_0^{-1} , \qquad R_{\rm in} := \delta H^{-1}_{-, {\rm in}} \oplus 0.
\ee
For $\delta >0$ such that $\Lambda_q (1+\delta)$ is not an eigenvalue of  $H_{-, {\rm in}}$, set
$$
r_{\rm in}:=  \delta^{\frac12} H_{-, {\rm in}}^{-\frac12} \Big( I- \Lambda_q (1+\delta) H^{-1}_{-, {\rm in}} \Big)^{-1} \oplus 0.
$$
Obviously,
$$
-R^{\frac12}_{\rm in} = \Lambda_q r_{\rm in}  \Big( (1 + \delta)  H^{-1}_{-, {\rm in}} \oplus H^{-1}_{-, {\rm ex}} - \Lambda_q^{-1} \Big) =
\Big( (1 + \delta)  H^{-1}_{-, {\rm in}} \oplus H^{-1}_{-, {\rm ex}} - \Lambda_q^{-1} \Big) \Lambda_q r_{\rm in},
$$
and since $ (1 + \delta)  H^{-1}_{-, {\rm in}} \oplus H^{-1}_{-, {\rm ex}}  =  V_{-,0} + H_0^{-1} $, we have
$$
R_{\rm in} = \Lambda_q^2 \Big( H_0^{-1} - \Lambda_q^{-1} +  V_{-,0}   \Big)  \, r^2_{\rm in} \, \Big( H_0^{-1} - \Lambda_q^{-1} +  V_{-,0} \Big).
$$
Moreover,
$$
 R_{\rm in} (H_0^{-1} - \Lambda_q^{-1}) H_0 \omega P_q (p_q \otimes \gp)=  R_{\rm in} (\omega -  \Lambda_q^{-1} \omega H_0) P_q (p_q \otimes \gp)+  R_{\rm in} [\omega,H_0] P_q (p_q \otimes \gp) =0.
 $$
Using this relation and the dual one, we deduce
$$
M^-_{4,q} = (p_q \otimes \gp)  P_q \omega  H_0  (V_{-,0} - R_{\rm in})  H_0 \omega P_q (p_q \otimes \gp)
$$
$$
= (p_q \otimes \gp) P_q \omega  H_0  \Big( V_{-,0} - \Lambda_q^2 V_{-,0}  r^2_{\rm in} V_{-,0}  \Big)  H_0 \omega P_q (p_q \otimes \gp) .
$$
By Lemma \ref{lm2},
    \bel{m31}
n_+(s;M^-_{4,q})
 \geq  n_+(2s; (p_q \otimes \gp) P_q \omega  H_0   V_{-,0} H_0 \omega P_q (p_q \otimes \gp) - {\rm Tr}\,\one_{[1/2,\infty)}(\Lambda_q^2 V_{-,0}  r^2_{\rm in} V_{-,0}) .
 \ee
Pick $\kappa \in (0,b)$,  and write
\bel{decV0}
V_{-,0}
= \Big(  (1 + \delta)  H^{-1}_{-, {\rm in}} \oplus H^{-1}_{-, {\rm ex}}  - ( H_0- \kappa \oneoless)^{-1} \Big) + \Big(  ( H_0- \kappa \oneoless)^{-1} - H_0^{-1} \Big),
\ee
the operator $V_{-,0}$ being defined in \eqref{defv0}.
Now choose $\kappa$ sufficiently small so that $\kappa (1+\delta) \delta^{-1}$ is smaller than the ground state of $H_{-, {\rm in}}$.
Then on $\gD((H_0- \kappa \oneoless)^{\frac12}) = {\rm H}_A^1(\re^3)$ the quadratic form of $\frac1{1 + \delta}  H_{-, {\rm in}} \oplus H_{-, {\rm ex}} $ is dominated by the one of $H_0- \kappa \oneoless$. More precisely, for any $u \in {\rm H}_A^1(\re^3)$ we have
$$
\|\Pi(A) u\|^2_{{\rm L}^2(\omex)} + \|\Pi(A) u\|^2_{{\rm L}^2(\omin)} - \kappa \|u\|^2_{{\rm L}^2(\Omega_<)}  \geq \|\Pi(A) u\|^2_{{\rm L}^2(\omex)} + \frac1{1 + \delta}  \|\Pi(A) u\|^2_{{\rm L}^2(\omin)}.
$$
By \cite[Proposition 2.1 (i)]{PuRo07}, this shows that the first term of \eqref{decV0} is a non-negative operator. Moreover, the resolvent identity implies
$$
\Big(( H_0- \kappa \oneoless)^{-1} -  H_0^{-1}  \Big)=  \kappa H_0^{-1} \oneoless  \Big( I + \kappa \oneoless ( H_0- \kappa \oneoless)^{-1}\oneoless  \Big) \oneoless H_0^{-1} \geq \kappa H_0^{-1} \oneoless  H_0^{-1}.
$$
Taking into account \eqref{decV0}, we get
$$
V_{-,0} \geq \Big( ( H_0- \kappa \oneoless)^{-1} -  H_0^{-1}  \Big) \geq \kappa H_0^{-1} \oneoless  H_0^{-1},$$
and the mini-max principle implies
    \bel{m32}
    n_+(s; (p_q \otimes \gp) P_q \omega  H_0   V_{-,0} H_0 \omega P_q (p_q \otimes \gp) \geq n_+(s; \kappa (p_q \otimes \gp) P_q \omega  \oneoless \omega P_q (p_q \otimes \gp)), \quad s>0.
    \ee
Finally, taking into account \eqref{m27a}, by analogy with \eqref{m27}, we obtain
    \bel{m33}
    n_+\left(s; (p_q \otimes \gp) P_q \omega  \oneoless \omega P_q (p_q \otimes \gp)\right) = n_+\left(\|\omega_4\|^{2}_{L^2(\re)}s;  p_q w_< p_q)\right), \quad s>0.
    \ee
    Now \eqref{m30} follows from \eqref{m31}, \eqref{m32}, and \eqref{m33}.
    \epf

    \subsection{Upper bounds of $n_+(s; \cM_q^\pm)$}
\label{ss2}

\beprl{3a}
 Let $q \in \Z_+$. Then there exist constants $C_q^\pm > 0$ such that
 \bel{m2}
 n_+(s;M_{4,q}^\pm) \leq n_+(C_q^\pm s; p_0 \, \one_{\pipe({\rm supp}\,\omega)} \, p_0) + O(1), \quad s>0.
    \ee
    \epr
    \bepf
    Evidently,
    $$
    M^\pm_{4,q} \leq \|V_\pm\| M_{5,q}
    $$
    where
    $$
    M_{5,q} : = (p_q \otimes \gp) \cT_q^*  \cT_q (p_q \otimes \gp).
    $$
    Similarly to $M^\pm_{4,q}$, the operator $M_{5,q}$ will be considered as a compact self-adjoint operator in the Hilbert space $(p_q \otimes \gp) L^2(\re^3)$. Then,
    \bel{m5}
    n_+(s;M^\pm_{4,q}) \leq  n_+(s;\|V_\pm\| M_{5,q}), \quad s>0.
    \ee


    Set $\omega_0 : = \Lambda_q \omega - \Delta \omega$. Define the operator
    $$
    M_{6,q} : = p_q \left(\sum_{j,k=0}^2 g_j^* w_{jk} g_k\right) p_q
    $$
    where
    $$
    g_0 = \Ipe, \quad g_1 = a^*, \quad g_2 = a,
    $$
    and
    $$
    w_{jk}(\xpe) : = \int_{\re} \overline{\omega_j}(\xpe,\xpa) \omega_k(\xpe,\xpa) d\xpa, \quad \xpe \in \rd, \quad j,k=0,1,2.
    $$
    Then the operator $M_{5,q}$ with domain $(p_q \otimes \gp) L^2(\re^3)$ is unitarily equivalent to the operator $\|\omega_4\|^2_{L^2(\re)} M_{6,q}$ with domain $p_q L^2(\rd)$. Therefore,
    \bel{m7}
    n_+(s;M_{5,q}) = n_+\left(s\|\omega_4\|^{-2}_{L^2(\re)}; M_{6,q}\right), \quad s>0.
    \ee
    In the Appendix \ref{as} we show that the operator $M_{6,q}$ is unitarily equivalent to
    \bel{f5}
    M_{7,q} : = p_0 \upsilon_q p_0,
    \ee
     where $\upsilon_q : \rd \to \re$ is an appropriate bounded multiplier so that $M_{7,q}$ is self-adjoint on its domain $p_0 L^2(\rd)$. More precisely, if $q \geq 1$, we have
     \begin{align} \label{f14}
     \upsilon_q &: = {\rm L}_{q}\left(-\frac{\Delta}{2b}\right) w_{00} + 2b(q+1) {\rm L}_{q+1}\left(-\frac{\Delta}{2b}\right) w_{11} + 2bq {\rm L}_{q-1}\left(-\frac{\Delta}{2b}\right) w_{22}\\[3mm] \nonumber
    &- 8 {\rm Re}\,{\rm L}_{q-1}^{(2)}\left(-\frac{\Delta}{2b}\right)\frac{\partial^2 w_{21}}{\partial \zeta^2}
    - 4 {\rm Im}\,{\rm L}_{q}^{(1)}\left(-\frac{\Delta}{2b}\right)\frac{\partial w_{01}}{\partial \zeta}
    - 4 {\rm Im}\,{\rm L}_{q-1}^{(1)}\left(-\frac{\Delta}{2b}\right)\frac{\partial w_{20}}{\partial \zeta},
    \end{align}
    where
    $$
      {\rm L}_q^{(m)}(t) : = \sum_{j=0}^q \binom{q+m}{q-j} \frac{(-t)^j}{j!}, \quad t \in \re, \quad q \in \Z_+, \quad m \in \Z_+,
      $$
     are the Laguerre polynomials;
      as usual, we write ${\rm L}_q^{(0)} = {\rm L}_q$. If $q=0$, then
      \bel{f16}
      \upsilon_0 : =  w_{00} + 2b {\rm L}_{1}\left(-\frac{\Delta}{2b}\right) w_{11}  - 4 {\rm Im}\,\frac{\partial w_{01}}{\partial \zeta} .
      \ee
      Therefore,
      \bel{f18}
      n_+(s; M_{6,q}) = n_+(s; M_{7,q}), \quad s>0.
      \ee
      Note that $\upsilon_q \in C_0^\infty(\rd;\re)$ and we have
      $$
      \mu_q : = \max_{\xpe \in \rd} \upsilon_q(\xpe) \in (0,\infty).
      $$
      Indeed, if $\upsilon_q \leq 0$, then \eqref{m4}, \eqref{m5}, and \eqref{f18} would imply that $M_{4,q}^\pm \leq 0$ which is impossible by Propositions  \ref{Sy1} and \ref{Sy2}.
      Moreover, it is easy to check that
      $$
      {\rm supp}\,\upsilon_q \subset \pipe(\rm supp\,\omega).
      $$
      Therefore,
      $$
      M_{7,q} \leq  \mu_q \, p_0 \,\one_{\pipe(\rm supp\,\omega)} \, p_0,
      $$
      and, hence,
      \bel{m6}
      n_+(s; M_{7,q}) \leq  n_+\left(s; \mu_q \, p_0 \, \one_{\pipe(\rm supp\,\omega)} \, p_0\right), \quad s>0.
      \ee
      Now \eqref{m1} follows from \eqref{m5}, \eqref{m7}, \eqref{f18}, and \eqref{m6}.
     \epf
    \subsection{Properties of the logarithmic capacity}
\label{ss55}
First, we list several elementary properties of the logarithmic capacity (see e.g. \cite[Chapter 5]{Ran95}):\\
(i) Let $\cE_1, \cE_2$ be Borel subsets of $\rd$ such that $\cE_1 \subset \cE_2$. Then, evidently,
$$
    \capa(\cE_1) \leq \capa(\cE_2).
    $$
    (ii) Let $\cK \subset \rd$ be a compact set. For $\delta > 0$, put
    $$
    \cK_\delta : = \left\{\xp \in \rd \, | \, {\rm dist}\,(\xpe, \cK) \leq \delta\right\}.
    $$
    Then we have
    \bel{au42}
    \lim_{\delta \downarrow 0} \capa(\cK_\delta) = \capa(\cK).
    \ee
    Next, we formulate a result which allows to approximate the logarithmic capacity of a bounded plane domain by the logarithmic capacities of curves contained in the domain. Let $\gamma \subset \rd$ be a Jordan curve, i.e. a simple closed curve. We will say that $\gamma$ is $C^2$-smooth if there exists a $C^2$-smooth diffeomorphism ${\bf x} : \Sbb^1 \to \gamma.$
    \begin{proposition} \label{pfin10} {\rm \cite[Proposition 5.6]{CaRaTe19}} Let $\cD \subset \rd$ be a bounded domain. Then there exists a sequence $\left\{\gamma_j\right\}_{j \in \N}$ of $C^2$-smooth Jordan curves such that $\gamma_j \subset \cD$ and
    \bel{fin61}
    \lim_{j \to \infty} \capa(\gamma_j) = \capa(\overline{\cD}).
    \ee
    \end{proposition}
    \begin{corollary} \label{ffin10}
    Let $\cD \subset \rd$ be a bounded domain. Then
    \bel{fin62}
    \capa(\cD) = \capa(\overline{\cD}).
    \ee
    \end{corollary}
    \begin{proof}
    Let $\left\{\gamma_j\right\}_{j \in \N}$ be the sequence of curves introduced in Proposition \ref{pfin10}. Then \eqref{fin62} follows immediately from $\gamma_j \subset \cD \subset \overline{\cD}$ and \eqref{fin61}.
    \end{proof}
\subsection{Eigenvalue asymptotics for the Toeplitz operators $p_q \one_\cO p_q$}
\label{ss56}
Let $\cO \subset \rd$ be a bounded domain. Fix $q \in \Z_+$. In the Hilbert space $p_q L^2(\rd)$, consider the operator $p_q \one_\cO p_q$ which is self-adjoint, compact, and non-negative. Moreover, the results of \cite[Subsection 4.3]{RaWa02} imply that all the eigenvalues of ${\rm rank}\,p_q \one_\cO p_q$ are strictly positive. Denote by $\left\{\nu_{k,q}(\cO)\right\}_{k \in \Z_+}$ the non-decreasing sequence of the eigenvalues of $p_q \one_\cO p_q$.

\begin{proposition} \label{fipu}
{\rm \cite[Lemma 2]{FiPu06}} Let $\cO \subset \rd$ be a bounded domain with Lipschitz boundary, and $q \in \Z_+$. Then we have
$$
\lim_{k \to \infty} \left(k! \nu_{k,q}(\cO)\right)^{1/k} = \frac{b \capa(\cO)^2}{2},
$$
or, equivalently,
    \bel{ms20}
    \ln{\nu_{k,q}}(\cO) = -k\,\ln{k} + \left(\gC(\cO) - \ln{2}\right) k + o(k), \quad k \to \infty,
    \ee
    where $\gC(\cO)$ is the constant defined in \eqref{ms50}.
    \end{proposition}

\begin{corollary} \label{msf1} Under the hypotheses of Proposition \ref{fipu} for any constant $c>0$ we have
    \bel{ms25}
    n_+(c\sqrt{\lambda}; p_q \one_\cO\,p_q) = \frac{1}{2}\, \Phi_1(\lambda; \gC(\cO)) + o\left(\frac{|\ln{\lambda}|}{\ln_2(\lambda)^2}\right), \quad \lambda \downarrow 0.
    \ee
     \end{corollary}
     The proof of the corollary can be found in Subsection \ref{as2} of the Appendix.\\

     \begin{corollary} \label{msf2} Under the hypotheses of Proposition \ref{fipu} for any constant $c>0$ we have
    \bel{ms25a}
    \frac{1}{\pi}{\rm Tr}\, \arctan{\left(\frac{p_q \one_\cO\,p_q}{c\sqrt{\lambda}}\right)} = \frac{1}{4}\, \Phi_1(\lambda; \gC(\cO)) + o\left(\frac{|\ln{\lambda}|}{\ln_2(\lambda)^2}\right), \quad \lambda \downarrow 0.
    \ee
     \end{corollary}
     The proof of the corollary is contained in Subsection \ref{as3} of the Appendix.\\
    \subsection{Proof of \eqref{m13} - \eqref{ms2}}
\label{ss57}
    For $\lambda > 0$ small enough,  and $q \in \Z_+$,  set
     $$
    \Xi_{q,1}^-(\lambda) : = \frac{-\xi(\Lambda_q-\lambda;H_-,H_0) - 2^{-1} \Phi_1(\lambda; 1+\ln{b})}{2^{-1} \Phi_0(\lambda) \lnd(\lambda)^{-1}},
    $$
    $$
    \Xi_{q,2}^\pm(\lambda) : = \frac{\pm\xi(\Lambda_q+\lambda;H_\pm,H_0) - 2^{-2} \Phi_1(\lambda; 1+\ln{b})}{2^{-2} \Phi_0(\lambda) \lnd(\lambda)^{-1}}.
    $$
    Then \eqref{m13} is equivalent to
    \bel{fin63}
    \lim_{\lambda \downarrow 0} \Xi_{q,1}^-(\lambda) = \ln{\capa(\cO_{\rm in})^2},
    \ee
    while \eqref{ms2} is equivalent to
    \bel{fin64}
    \lim_{\lambda \downarrow 0} \Xi_{q,2}^\pm(\lambda) = \ln{\capa(\cO_{\rm in})^2}.
    \ee
    Let us first prove \eqref{fin63}, starting with the corresponding lower asymptotic  bound.
    Combining \eqref{f7} and \eqref{m4} with \eqref{m30}, we find that for each domain $\Omega_<$ such that $\overline{\Omega_<} \subset \Omega_{\rm in}$ there exists a constant $c>0$ such that
     \bel{ms57a}
     -\xi(\Lambda_q-\lambda;H_-,H_0) \geq n_+(c\sqrt{\lambda}; p_q w_{\Omega_<} p_q) + O(1), \quad \lambda \downarrow 0,
     \ee
     where $w_{\Omega_<}$ is the function defined in \eqref{fin60}. Let us construct a suitable sequence of domains compactly embedded in $\Omega_{\rm in}$.
     Let
     $$
     \gamma_j^< : = \left\{{\bf x}_j(s) \, | \, s \in \Sbb^1\right\}, \quad j \in \N,
     $$
     be a sequence of $C^2$-smooth Jordan curves such that $\gamma_j^< \subset \cO_{\rm in}$ and
     \bel{fin65}
     \lim_{j \to \infty} \capa(\gamma_j^<) = \capa(\overline{\cO_{\rm in}}) = \capa(\cO_{\rm in}),
     \ee
    whose existence is guaranteed by Proposition \ref{pfin10}. Let
     $$
     {\bf n}_j(s) : = \frac{(x'_{2,j}(s), -x'_{1,j}(s))}{|{\bf x}_j'(s)|}, \quad s \in \Sbb^1,
     $$
     be the normal unit at $\gamma_j^<$, $j \in \N$. For $\tau \in (0,1]$ set
     $$
     \cO_{j,\tau}^< : = \left\{{\bf x}_j(s) + t {\bf n}_j(s) \, | \, s \in \Sbb^1, \; |t| < \tau \varepsilon_j\right\}
     $$
     where $\varepsilon_j > 0$, $j \in \N$, is chosen so small that $ \overline{\cO_{j,1}^<} \subset \cO_{\rm in} $ and the boundary $\partial \cO_{j,1}^<$ is Lipschitz. Then, evidently, the domains $ \cO_{j,\tau}^<$ with $\tau \in (0,1)$ have the same properties. Set
     $$
     \Omega_{j,\tau} : = \left\{(\xpe,\xpa) \in \Omega_{\rm in} \, | \, \xpe \in \cO_{j,\tau}^<\right\}, \quad j \in \N, \quad \tau \in (0,1].
     $$
     Evidently, $\Omega_{j,\tau} $ is a domain and $\overline{\Omega_{j,\tau}} \subset \Omega_{\rm in}$. Since $ w_{\Omega_{j,\tau}}(\xpe) \geq 0$ for every $\xpe \in \rd$ and $ w_{\Omega_{j,\tau}}(\xpe) > 0$ if and only if $\xpe \in \cO_{j,\tau}^<$, we have
     \bel{fin80}
     w_{\Omega_{j,\tau}} \geq \one_{\overline{\cO^<_{j,\tau/2}}} \,w_{\Omega_{j,\tau}}  \geq c_1  \one_{\overline{\cO^<_{j,\tau/2}}}
     \ee
     where
     $$
     c_1 : = \inf_{\xpe \in \overline{\Omega_{j,\tau/2}}} w_{\Omega_{j,\tau}}(\xpe).
      $$
      Since
      $$
       w_{\Omega_{j,\tau}}(\xpe) = \left\{
       \begin{array} {l}
       w_{\Omega_{\rm in}}(\xpe) \quad {\rm if} \quad \xpe \in \cO^<_{j,\tau},\\[2mm]
       0 \quad {\rm if} \quad \xpe \in \rd \setminus \cO^<_{j,\tau},
       \end{array}
       \right.
       $$
       we have
       \bel{fin80a}
     c_1  = \inf_{\xpe \in \overline{\Omega_{j,\tau/2}}} w_{\Omega_{\rm in}}(\xpe).
     \ee
     Let us now show that if $\cK \subset \cO_{\rm in}$ is a compact set, then
     \bel{fin80b}
     \inf_{\xpe \in \cK} w_{\Omega_{\rm in}}(\xpe) > 0.
     \ee

     Let $\ype \in \cO_{\rm in}$. Then there exists $\ypa \in \re$ such that $y : = (\ype, \ypa) \in \Omega_{\rm in}$.
     Since $\Omega_{\rm in}$ is open, there exists $r = r(\ype)>0$ such that $\cB_r(\ype) \times (\ypa-r,\ypa+r) \subset \Omega_{\rm in}$
     where $\cB_r(\ype): = \left\{\xpe \in \rd \, | \, |\xpe - \ype| < r\right\}$. Therefore, for every $\ype \in \cO_{\rm in}$ there exists
     $r = r(\ype)>0$ such that for each $\xpe \in \cB_r(\ype)$ we have
     $$
     w_{\Omega_{\rm in}}(\xpe) = \int_\re\one_{\Omega_{\rm in}}(\xpe,t) dt \geq \int_{\ypa-r}^{\ypa+r}\,dt = 2r.
     $$
     Since $\cK$ is a compact subset of $\cO_{\rm in}$, there exists a finite set $\left\{y_{\perp, j}\right\}_{j =1}^J \subset \cO_{\rm in}$ with $J \in \N$ such that
     $\cK \subset \cup_{j = 1}^J \cB_{r(y_{\perp, j})}(y_{\perp, j})$. Set $\rho : = \min_{j=1,\ldots J} \, r(y_{\perp.j})$. Then we have
     $$
     w_{\Omega_{\rm in}}(\xpe) \geq 2 \rho > 0, \quad \xpe \in \cK,
     $$
     which implies \eqref{fin80b}. By \eqref{fin80a} and \eqref{fin80b}, we obtain $c_1>0$.
     Therefore, \eqref{fin80} and the mini-max principle yield
     \bel{fin66}
     n_+(s; p_q w_{\Omega_{j,\tau}}p_q) \geq  n_+(s; c_1 p_q \one_{\overline{\cO^<_{j,\tau/2}}} p_q) =  n_+(c_1^{-1}s; p_q \one_{\cO^<_{j,\tau/2}} p_q), \quad s>0.
     \ee
     Now, \eqref{ms57a}, \eqref{fin66}, and \eqref{ms25} imply
    \bel{fin67}
    \liminf_{\lambda \downarrow 0} \Xi_{q,1}^-(\lambda) \geq \ln{\capa(\cO^<_{j, \tau/2})^2}.
    \ee
    By $\gamma_j^< \subset \cO^<_{j, \tau/2} \subset \overline{\cO_{\rm in}}$ and \eqref{fin65}, we have
    \bel{fin73}
    \lim_{j \to \infty} \capa(\cO^<_{j, \tau/2}) = \capa(\overline{\cO_{\rm in}}) = \capa(\cO_{\rm in}),
    \ee
    which combined with \eqref{fin67} yields
    \bel{fin68}
    \liminf_{\lambda \downarrow 0} \Xi_{q,1}^-(\lambda) \geq \ln{\capa(\cO_{\rm in})^2}.
    \ee
    Let us now estimate $\Xi_{q,1}^-(\lambda)$ from above.
     Combining \eqref{f7} and \eqref{m4} with \eqref{m2}, we find that for each
     $\omega \in C_0^\infty(\rt;\re)$ satisfying $\omega = 1$  on $\overline{\Omega_{\rm in}}$,
      there exists a constant $c>0$ such that
     \bel{fin69}
     -\xi(\Lambda_q-\lambda;H_-,H_0) \leq n_+(c\sqrt{\lambda}; p_0 \, \one_{\pipe({\rm supp}\, \omega)} \, p_0) + O(1), \quad \lambda \downarrow 0.
     \ee
      For $\delta > 0$ small enough set
      $$
      \Omega_\delta : = \left\{x \in \re^3 \, | \, {\rm dist}\,(x,\Omega_{\rm in}) \leq \delta \right\}, \quad
      \cO_\delta : = \left\{\xpe \in \re^2 \, | \, {\rm dist}\,(\xpe,\cO_{\rm in}) \leq \delta \right\},
      $$
      and choose $\omega$ so that ${\rm supp}\,\omega = \Omega_\delta$. Then we have
      $$
      \pipe({\rm supp}\, \omega) = \pipe(\Omega_\delta) \subset \cO_\delta.
      $$
      In order to check the above inclusion, assume that $\xpe \in \pipe(\Omega_\delta)$. Then there exists $x \in \Omega_\delta$ such that
      $\pipe(x) = \xpe$ and $y \in \overline{\Omega_{\rm in}}$ satisfying
      $$
      {\rm dist}\,(x,\Omega_{\rm in}) = {\rm dist}\,(x,\overline{\Omega_{\rm in}}) = |x-y| \leq \delta.
      $$
      Let
      $\ype : = \pipe(y) \in \pipe(\overline{\Omega_{\rm in}}) = \overline{\cO_{\rm in}}$.
      We have
      $$
      |\xpe - \ype| \leq |x-y| \leq \delta.
      $$
      Therefore,
      $$
      {\rm dist}\,(\xpe,\cO_{\rm in}) = {\rm dist}\,(\xpe,\overline{\cO_{\rm in}}) \leq |\xpe-\ype| \leq \delta,
      $$
      and, hence, $\xpe \in \cO_\delta$. \\
     Since $\cO_\delta$ is compact, there exist finite coverings of $\cO_\delta$ by squares with sides parallel to the coordinate axes, of arbitrarily small size. Hence, there exists a domain $\cO_{\delta}^> \subset \rd$ with Lipschitz boundary such that
     \bel{fin70}
     \overline{\cO_{\rm in}} \subset \cO_\delta \subset \cO_{\delta}^>, \quad \overline{\cO_{\delta}^>} \subset \cO_{2\delta}.
     \ee
       Then \eqref{fin69} and the mini-max principle imply
    $$
     -\xi(\Lambda_q-\lambda;H_-,H_0) \leq n_+(c\sqrt{\lambda}; p_0 \, \one_{\cO_{\delta}^>} \, p_0) + O(1), \quad \lambda \downarrow 0,
     $$
     which combined with \eqref{ms25} yields
     \bel{fin71}
     \limsup_{\lambda \downarrow 0} \Xi_{q,1}^-(\lambda) \leq \ln{\capa(\cO^>_{\delta})^2}.
    \ee
    By \eqref{fin70}, \eqref{au42}, and \eqref{fin62}, we have
    \bel{fin75}
    \lim_{\delta \downarrow 0} \capa(\cO_{\delta}^>)  = \capa(\cO_{\rm in}).
    \ee
    Therefore, \eqref{fin71} and \eqref{fin75} imply
    \bel{fin72}
    \limsup_{\lambda \downarrow 0} \Xi_{q,1}^-(\lambda) \leq \ln{\capa(\cO_{\rm in})^2}.
    \ee
    Putting together \eqref{fin68} and \eqref{fin72}, we obtain \eqref{fin63}.\\
     The proof of \eqref{fin64} is quite similar. Note that for any trace-class operator $T = T^* \geq 0$, we have
     \bel{sep1}
     {\rm Tr}\,\arctan{T} = \int_0^\infty n_+(s;T) \frac{ds}{1+s^2}.
     \ee

     Combining \eqref{f8}-\eqref{f9} and \eqref{m4} with \eqref{m22}
      or \eqref{m30}, and \eqref{m2}, and bearing in mind \eqref{sep1} and the mini-max principle, we find that there exist constants $c_< \geq   c_> > 0$ such that
      $$
     \frac{1}{\pi} {\rm Tr}\,\arctan{\left(\frac{c_1 p_q \one_{\cO^<_{j,\tau/2}} p_q}{c_< \sqrt{\lambda}}\right)} + O(1) \leq
     $$
     $$
     \pm\xi(\Lambda_q+\lambda;H_\pm,H_0) \leq
     $$
     \bel{ms61}
      \frac{1}{\pi} {\rm Tr}\,\arctan{\left(\frac{p_0 \one_{\pipe({\rm supp}\,\omega}) p_0}{c_> \sqrt{\lambda}}\right)} + O(1), \quad \lambda \downarrow 0.
     \ee
    Putting together \eqref{ms61} and \eqref{ms25a}, we get
     $$
     \ln{\capa(\cO^<_{j,\tau/2})^2} \leq \liminf_{\lambda \downarrow 0} \Xi_{q,2}^\pm(\lambda) \leq \limsup_{\lambda \downarrow 0} \Xi_{q,2}^\pm(\lambda) \leq
     \ln{\capa(\cO^>_{\delta})^2},
     $$
     which together with \eqref{fin73} and \eqref{fin75}, implies \eqref{fin64}.
\appendix
\section{}
\label{as}
\subsection{Unitary equivalence of $M_{6,q}$ and $M_{7,q}$}
\label{a1s}
In our proof of the unitary equivalence of the operators $M_{6,q}$ and $M_{7,q}$ (see Subsection \ref{ss2}) we will follow closely the argument of the proof of \cite[Proposition 4.1]{LuRa15}. Set
       \bel{au73}
       \varphi_{k,0}(\xpe) : = \sqrt{\frac{b}{2\pi k!}} \left(\frac{b}{2}\right)^{k/2} \zeta^k e^{-b|x|^2/4}, \quad \xpe \in \rd, \quad k \in \Z_+,
       \ee
        \bel{au74}
       \varphi_{k,q}(x) : = \sqrt{\frac{1}{(2b)^q q!}}(a^*)^q\varphi_{k,0}(x), \quad \xpe \in \rd, \quad k \in \Z_+, \quad q \in {\mathbb N}.
       \ee
       Then $\left\{\varphi_{k,q}\right\}_{k \in \Z_+}$ is an orthonormal basis of $p_q L^2(\rd)$ called sometimes {\em the angular momentum basis}
       (see e.g. \cite{RaWa02} or \cite[Subsection 9.1]{BrPuRa04}). Evidently, for $k \in \Z_+$ we have
       \bel{37}
       a^* \varphi_{k,q} = \sqrt{2b(q+1)}  \varphi_{k,q+1}, \quad q \in \Z_+,
       \quad a \varphi_{k,q} = \left\{
       \begin{array} {l}
       \sqrt{2bq}  \varphi_{k,q-1}, \quad q \geq 1,\\
       0 , \quad q = 0.
       \end{array}
       \right.
       \ee
       Define the unitary operator ${\mathcal W}: p_q L^2(\rd) \to p_0 L^2(\rd)$ by ${\mathcal W} : u \mapsto v$
       where
       $$
       u = \sum_{k  \in \Z_+} c_k \varphi_{k,q}, \quad v = \sum_{k  \in \Z_+} c_k \varphi_{0,k}, \quad \{c_k\}_{k \in \Z_+} \in \ell^2(\Z_+).
       $$
       We will show that
       \bel{36}
       M_{6,q} = {\mathcal W}^* M_{7,q} {\mathcal W}.
       \ee
       For $V \in C_{\rm b}^{\infty}(\rd)$, $m,s \in \Z_+$, and $k,\ell \in \Z_+$, set
       $$
       \Upsilon_{m,s}(V;k,\ell) : = \langle V \varphi_{k,m}, \varphi_{\ell,s}\rangle
       $$
       where $\langle\cdot,\cdot\rangle$ denotes the scalar product in $L^2(\rd)$. Taking into account \eqref{37}, we easily find that
       \bel{f15}
       \langle M_{6,q} u, u\rangle =
       \ee
       $$
        \sum_{k \in \Z_+} \sum_{\ell \in \Z_+} \left(\Upsilon_{q,q}(w_{00};k,\ell) + 2b(q+1)\Upsilon_{q+1,q+1}(w_{11};k,\ell) + 2bq\Upsilon_{q-1,q-1}(w_{22};k,\ell)\right)c_k \overline{c_\ell} \;  +
       $$
       $$
       2{\rm Re} \, \sum_{k \in \Z_+} \sum_{\ell \in \Z_+}\left(2b\sqrt{q(q+1)}\Upsilon_{q+1,q-1}(w_{21};k,\ell)  \right)c_k \overline{c_\ell} \; +
       $$
       $$
       2{\rm Re} \, \sum_{k \in \Z_+} \sum_{\ell \in \Z_+}\left( \sqrt{2b(q+1)} \Upsilon_{q+1,q}(w_{01};k,\ell) + \sqrt{2bq} \Upsilon_{q,q-1}(w_{20};k,\ell)\right)c_k \overline{c_\ell},
       $$
       if $q \geq 1$, and
      \bel{41}
       \langle M_{6,0}  u, u\rangle = \sum_{k \in \Z_+} \sum_{\ell \in \Z_+} \left(\Upsilon_{0,0}(w_{00};k,\ell) + 2b\Upsilon_{1,1}(\omega_{11};k,\ell)\right) c_k \overline{c_\ell} \; +
       \ee
        $$
        +  2\sqrt{2b}\,{\rm Re} \, \sum_{k \in \Z_+} \sum_{\ell \in \Z_+} \, \Upsilon_{1,0}(w_{01};k,\ell) \, c_k \overline{c_\ell}.
       $$
       Moreover,
       \bel{42}
       \langle M_{7,q} \cW u, \cW u\rangle = \sum_{k \in \Z_+} \sum_{\ell \in \Z_+} \Upsilon_{0,0}(\upsilon_q;k,\ell) c_k \overline{c_\ell}, \quad q \in \Z_+.
       \ee
       In \cite[Lemma 9.2]{BrPuRa04} (see also the remark after Eq.(2.2) in \cite{BoBrRa14b}), it was shown that
       \bel{40}
       \Upsilon_{m,m}(V;k,\ell) = \Upsilon_{0,0}\left({\rm L}_m\left(-\frac{\Delta}{2b}\right)V;k,\ell\right), \quad m \in \Z_+,
       \ee
       for any $V \in C_{\rm b}^\infty(\rd)$. Moreover, by \cite[Eq. (4.27)]{LuRa15}, we have
        \bel{55}
       2b\sqrt{q(q+1)} \Upsilon_{q+1,q-1}(V;k,\ell) = \Upsilon_{0,0}\left(-4{\rm L}_{q-1}^{(2)}\left(-\frac{\Delta}{2b}\right)\frac{\partial^2 V}{\partial z^2}; k,\ell\right).
       \ee
       It remains to handle the quantity $\Upsilon_{q+1,q}(V;k,\ell)$ with $q \in \Z_+$. We have
       \bel{f11}
       \Upsilon_{q+1,q}(V;k,\ell) =  \frac{1}{\sqrt{2b(q+1)}} \Upsilon_{q,q}([V,a^*];k,\ell) + \sqrt{\frac{q}{q+1}}  \Upsilon_{q,q+1}(V;k,\ell).
       \ee
       By \eqref{40} and \eqref{f11}, it is not difficult to show by induction that
       \bel{f12}
       \sqrt{2b(q+1)}\,\Upsilon_{q+1,q}(V;k,\ell) = \sum_{j=0}^q \Upsilon_{j,j}([V,a^*];k,l).
       \ee
       Taking into account \eqref{40}, as well as the facts that
       $$
       [V,a^*] = 2i \frac{\partial V}{\partial \zeta},
       $$
       by \eqref{f19}, and that
      $$
      \sum_{j=0}^q {\rm L}_j^{(m)}(t) = {\rm L}_q^{(m+1)}(t), \quad t \in \re, \quad q \in \Z_+, \quad m \in \Z_+,
       $$
       by \cite[Eq. 8.974.3]{GrRy65}, we find that \eqref{f12} implies
       \bel{f13}
       \sqrt{2b(q+1)}\,\Upsilon_{q+1,q}(V;k,\ell) = 2i\Upsilon_{0,0}\left({\rm L}_q^{(1)}\left(-\frac{\Delta}{2b}\right)\frac{\partial V}{\partial \zeta};k,l\right).
       \ee
      By \eqref{49}, \eqref{55}, \eqref{f13}, and the definition \eqref{f14} -- \eqref{f16} of $\upsilon_q$, we find that \eqref{f15}, \eqref{41}, and \eqref{42} imply \eqref{36}.
    \subsection{Proof of Corollary \ref{msf1}}
    \label{as2}
    First of all, we note that  elementary calculations show that for any constants $c>0$ and $C \in \re$ we have
    \bel{ms26}
    \Phi_1(c\lambda; C) =  \Phi_1(\lambda; C) + o(1),
    \ee
    and
    \bel{ms29}
    \Phi_1(\sqrt{\lambda}; C) =  \frac{1}{2}\,\Phi_1(\lambda; C + \ln{2}) + o\left(\frac{|\ln{\lambda}|}{\ln_2(\lambda)^2}\right),
    \ee
    as $\lambda \downarrow 0$. Further, by definition,
    $$
    n_+(\lambda; p_q \one_\cO\,p_q) = \#\,\left\{k \in \Z_+ \, | \, \nu_{k,q} > \lambda\right\}, \quad \lambda>0.
    $$
    Therefore,
    \bel{ms21}
    n_+(\lambda; p_q \one_\cO\,p_q) = \#\,\left\{k \in \Z_+ \, | \, |\ln{\nu_{k,q}}| < |\ln{\lambda}|\right\} + O(1), \quad \lambda \downarrow 0.
    \ee
    Combining \eqref{ms20} and \eqref{ms21}, we find that for every $\veps > 0$ we have
    $$
    \#\,\left\{k \in \Z_+ \, | \, k\ln{k} - (\gC(\cO) - \ln{2} - \veps)k < |\ln{\lambda}|\right\} + O(1) \leq
    $$
    $$
    n_+(\lambda; p_q \one_\cO\,p_q) \leq
    $$
    \bel{ms22}
    \leq \#\,\left\{k \in \Z_+ \, | \, k\ln{k} - (\gC(\cO) - \ln{2} + \veps)k < |\ln{\lambda}|\right\} + O(1),
    \ee
    as $\lambda \downarrow 0$. For $C \in \re$ set
    $$
    F_C(x) : = x \ln{x} - Cx, \quad x>0.
    $$
    Note that $F_C'(x) > 0$ if $x>e^{C-1}$. Hence,
    \bel{ms23}
    \leq \#\,\left\{k \in \Z_+ \, | \, k\ln{k} - Ck < |\ln{\lambda}|\right\} = F_C^{-1}(|\ln{\lambda}|) + O(1), \quad \lambda \downarrow 0.
    \ee
    We have
    \bel{ms24}
     F_C^{-1}(y) = \frac{y}{\ln{y}} + \frac{y\,\ln{\ln{y}}}{(\ln{y})^2} + \frac{Cy}{(\ln{y})^2} + o\left(\frac{y}{(\ln{y})^2}\right), \quad y \to \infty.
     \ee
     To see this, set $u = u(y) : = \frac{\ln{y}}{y}  F_C^{-1}(y) -1$ for $y>0$ large enough. Then $u$ satisfies the equation
     $$
     u = \frac{\ln{\ln{y}}}{\ln{y}} + \frac{C}{\ln{y}} + \left(\frac{\ln{\ln{y}}}{\ln{y}} + \frac{C}{\ln{y}}\right)u  - (1+u) \frac{\ln{(1+u)}}{\ln{y}}.
     $$
     Applying a suitable version of the contraction mapping principle, we find that
     $$
     u(y) = \frac{\ln{\ln{y}}}{\ln{y}} + \frac{C}{\ln{y}} + o\left( \frac{1}{\ln{y}}\right), \quad y \to \infty,
     $$
     which implies \eqref{ms24}. Putting together \eqref{ms22}, \eqref{ms23}, and \eqref{ms24}, we get
     \bel{ms32}
    n_+(\lambda; p_q \one_\cO\,p_q) = \Phi_1(\lambda; \gC(\cO) - \ln{2}) + o\left(\frac{|\ln{\lambda}|}{\ln_2(\lambda)^2}\right), \quad \lambda \downarrow 0.
    \ee
    Bearing in mind \eqref{ms26} and \eqref{ms29}, we conclude that \eqref{ms32} implies \eqref{ms25}.
    \subsection{Proof of Corollary \ref{msf2}}
    \label{as3}
     First of all, we note that similarly to \eqref{ms26}  for any constant $C \in \re$ we have
     \bel{ms27}
    \Phi_1(\lambda \, |\ln{\lambda}|; C) =  \Phi_1(\lambda; C) + O(1),
    \ee
    \bel{ms28}
    \Phi_1(\lambda \, |\ln{\lambda}|^{-1}; C) =  \Phi_1(\lambda; C) + O(1),
    \ee
    as $\lambda \downarrow 0$. Further, \eqref{sep1} yields
    \bel{ms20a}
    {\rm Tr}\,\arctan\left(\lambda^{-1} p_q \one_\cO\,p_q\right)
    = \int_\re \frac{n_+(\lambda t; p_q \one_\cO\,p_q)}{1+t^2} dt, \quad \lambda>0.
    \ee
    Let us estimate from above the integral $\int_\re \frac{n_+(\lambda t; p_q \one_\cO\,p_q)}{1+t^2} dt$ with $\lambda > 0$ small enough.
    Taking into account \eqref{ms32}, \eqref{ms28}, and the fact that the function $n_+(\cdot; p_q\one_\cO\,p_q)$ is non-increasing, we find that for any
    $\veps>0$ we have
   $$
    \int_\re \frac{n_+(\lambda t; p_q \one_\cO\,p_q)}{1+t^2} dt =  \int_{|\ln{\lambda}|^{-1}}^\infty \frac{n_+(\lambda t; p_q \one_\cO\,p_q)}{1+t^2} dt
     + \int_0^{|\ln{\lambda}|^{-1}} \frac{n_+(\lambda t; p_q \one_\cO\,p_q)}{1+t^2} dt \leq
     $$
     $$
    n_+(\lambda |\ln{\lambda}|^{-1}; p_q \one_\cO\,p_q)\int_{|\ln{\lambda}|^{-1}}^\infty \frac{dt}{1+t^2}
     + \int_0^{|\ln{\lambda}|^{-1}} \Phi_1(\lambda t; \gC(\cO) - \ln{2} + \veps) dt =
     $$
      \bel{ms33}
      \frac{\pi}{2} \Phi_1(\lambda |\ln{\lambda}|^{-1}; \gC(\cO - \ln{2}) + o\left(\frac{|\ln{\lambda}|}{\ln_2(\lambda)^2}\right), \quad \lambda \downarrow 0,
      \ee
      where at the last line we have used that
    $$
    \int_0^{|\ln{\lambda}|^{-1}} \Phi_1(\lambda t; \gC(\cO) - \ln{2} + \veps) dt = o(1), \;
     n_+(\lambda |\ln{\lambda}|^{-1}; p_q \one_\cO\,p_q) \arctan(\ln{\lambda}|^{-1}) = o(1),
    $$
    as $\lambda \downarrow 0$.
    Let us now estimate from below $\int_\re \frac{n_+(\lambda t; p_q \one_\cO\,p_q)}{1+t^2} dt$ with small $\lambda > 0$. By \eqref{ms32} and \eqref{ms28}, we easily obtain
    $$
    \int_\re \frac{n_+(\lambda t; p_q \one_\cO\,p_q)}{1+t^2} dt
     \geq \int_0^{|\ln{\lambda}|} \frac{n_+(\lambda t; p_q \one_\cO\,p_q)}{1+t^2} dt \geq
     $$
     \bel{ms33a}
    n_+(\lambda |\ln{\lambda}|; p_q \one_\cO\,p_q)\int_0^{|\ln{\lambda}|} \frac{dt}{1+t^2}
       = \frac{\pi}{2} \Phi_1(\lambda; \gC(\cO)-\ln{2}) + o\left(\frac{|\ln{\lambda}|}{\ln_2(\lambda)^2}\right), \quad \lambda \downarrow 0.
      \ee
      Now, \eqref{ms20a}, \eqref{ms33}, and \eqref{ms33a} imply
      $$
      \frac{1}{\pi}\,{\rm Tr}\,\arctan\left(\lambda^{-1} p_q \one_\cO\,p_q\right) =
      \frac{1}{2} \Phi_1(\lambda; \gC(\cO)-\ln{2}) + o\left(\frac{|\ln{\lambda}|}{\ln_2(\lambda)^2}\right), \quad \lambda \downarrow 0,
      $$
      which combined with \eqref{ms26} and \eqref{ms29} yields \eqref{ms25a}.\\

{\bf Acknowledgements}. The authors gratefully acknowledge  the partial
support of the Chilean Science Foundation {\em Fondecyt} under Grant 1170816. \\
Considerable parts of this work were done during  the authors' visits to MSRI, Berkeley, in July 2017, the Oberwolfach Research Institute for Mathematics in February -- March, 2018, UNSW, Sydney, in July 2018, CIRM, Marseille, in January 2019, and during G. Raikov's visit to the Mittag-Leffler Institute, Stockholm, in February -- March 2019. The authors thank  all these institutions for hospitality and financial support.

{\sc Vincent Bruneau}\\ Institut de Math\'ematiques de Bordeaux, UMR  CNRS 5251\\
Universit\'e de Bordeaux\\ 351 cours de la Lib\'eration, 33405 Talence cedex, France\\
E-mail: vbruneau@math.u-bordeaux.fr\\

{\sc Georgi Raikov}\\
 Institute of Mathematics and Informatics\\
Bulgarian Academy of Sciences\\
Acad. G. Bonchev Str., bl. 8,
1113 Sofia, Bulgaria\\
{\em On leave of absence from}:\\
 Facultad de Matem\'aticas\\ Pontificia Universidad
Cat\'olica de Chile\\ Vicu\~na Mackenna 4860, Santiago de Chile\\
E-mail: graikov@mat.uc.cl

\end{document}